\documentclass[12pt]{amsart}
\usepackage{amsfonts,amssymb}
\usepackage{amsmath,amscd}
\usepackage[normalem]{ulem}
\usepackage{soul}
\usepackage[all]{xy}
\usepackage{color}

\theoremstyle{definition}
\newtheorem{rem}[subsection]{Remark}
\theoremstyle{plain}
\newtheorem{prop}[subsection]{Proposition}
\newtheorem{thm}[subsection]{Theorem}
\newtheorem{lem}[subsection]{Lemma}
\newtheorem{cor}[subsection]{Corollary}

\newcommand{\mbf}{\mathbf}
\newcommand{\mbb}{\mathbb}
\newcommand{\mrm}{\mathrm}
\newcommand{\mrk}{\mathfrak}
\newcommand{\A}{\mathcal A}

\newcommand{\D}{\mathcal D}

\newcommand{\F}{\mathcal F}

\newcommand{\Q}{\mbb Q}

\newcommand{\U}{\mbf U}
\newcommand{\tF}{\widetilde{\mathcal F}}

\title[Geometric  quantum $\mathfrak{osp}(1|2)$] {A geometric setting for quantum $\mathfrak{osp}(1|2)$}

\author{Zhaobing Fan and  Yiqiang Li}
\address{Department of Mathematics\\ University at Buffalo, SUNY\\244 Mathematics Building
\\Buffalo, NY 14260}
\email{
zhaobing@buffalo.edu (Z.Fan),
yiqiang@buffalo.edu (Y.Li)
}

\date{\today}
\keywords{Quantum $\mathfrak{osp}(1|2)$, quantum modified algebra,  tensor product module, categorification, perverse sheaf}
\subjclass{17B37, 14F43}

\begin{document}
\begin{abstract}
A geometric  categorification is given for  arbitrary-large-finite-dimensional quotients of  quantum  $\mathfrak{osp}(1|2)$ and  tensor products of its simple  modules.
The  modified quantum  $\mathfrak{osp}(1|2)$ of Clark-Wang, a new version in this paper and the modified quantum  $\mathfrak{sl}(2)$
are shown to be isomorphic to each other over a field containing  $\mathbb Q(v)$ and $\sqrt{-1}$.
\end{abstract}

\maketitle

\section{Introduction}\label{Introduction}

\subsection{}
In the classical work ~\cite{Lusztig90, Lusztig91, Lusztig93} of Lusztig, he gives a geometric construction of the negative half of the quantum  algebra associated to  a Kac-Moody Lie algebra. It is later shown by Vasserot and Varagnolo  in ~\cite{VV11} that the extension algebra of Lusztig's complexes is isomorphic to
the KLR algebra, a.k.a. quiver Hecke algebra,
of symmetric type  introduced independently by Khovanov-Lauda and Rouquier in ~\cite{Khovanov2009diag} and ~\cite{R08}.
The KLR algebras admit an odd/super analogue, the so-called quiver Hecke superalgebras  by Kang-Kashiwara-Tsuchioka ~\cite{KKT} (see also ~\cite{EKL11} and  ~\cite{W09}).
By using  representation theory of  quiver Hecke superalgebras,
Hill and Wang ~\cite{HillWang} give a categorification of the negative half of a covering algebra
involving two parameters  $(q, \pi)$,
which specializes to the negative half of a quantum algebra at $\pi =1 $ and that of a quantum superalgebra at $\pi =-1$.
See also ~\cite{EKL11, EL13, KKO12, KKO13} for further progress in this active research direction.

To this end, it is natural to ask if one can categorify the negative part of Hill-Wang's  covering algebra and, moreover, the covering algebra itself  (or its modified form) by using   representation theory of KLR alegbras, or equivalently from Lusztig's geometric setting.
This question is first raised by Weiqiang Wang  and answered affirmatively for the negative half of the covering algebra by the authors  in ~\cite{FL12} and ~\cite{CFLW} together with Clark and  Wang.
A new idea in answering this question is that the Tate twist (mod 4)   categorifies the square root of the second parameter $\pi $ in the covering algebra.

After the negative half of  the covering algebra is categorified in Lusztig's geometric setting for negatives halves of quantum algebras,
we are  forced intuitively to search for a categorification of  the covering algebra  in a geometric setting analogous to the one for the negative halves.
Indeed, there is a such setting for quantum $\mathfrak{osp}(1|2)$, one of the smallest quantum superalgebras.
It is in  Beilinson-Lusztig-MacPherson's  geometric construction \cite{BLM90}
of
the $q$-Schur algebra,  a quotient of a quantum algebra of type $A$.

As one of the main results in this paper, we give a geometric  construction of an arbitrary large  finite-dimensional quotient of  quantum  $\mathfrak{osp}(1|2)$, as well as tensor products of its highest weight modules.
We follow the approach taken in our previous work by adding the Tate twist, or the mixed structure, to the geometric setting laid out  in  ~\cite{BLM90} involving the geometry of two copies of Grassmannians.
With a slight, though non-trivial, modification of the generators for the geometric construction of quantum $\mathfrak{sl}(2)$ in ~\cite{BLM90},
we obtain a quotient of quantum $\mathfrak{osp}(1|2)$.
In this realization, the Tate twist corresponds to the imaginary unit $t=\sqrt{-1}$.
We tend to call this quotient the $q$-Schur algebra of quantum $\mathfrak{osp}(1|2)$ because it gets identified with that of quantum $\mathfrak{sl}(2)$ immediately from the construction. Along the way, we also obtain a geometric construction of  tensor products of finite dimensional simple modules of quantum $\mathfrak{osp}(1|2)$ following ~\cite{Zheng07} (see also ~\cite{GL92})
using perverse sheaves on Grassmannians.

Just like the authors' previous work \cite{FL12}, the simple perverse sheaves of weight zero arising from this construction form a basis for the categorified quotients and tensor product modules.
The structure constants with respect to this basis possess again a positivity property in an appropriate sense (see Theorem \ref{thm1}).
We provide with an algebraic  characterization, up to a sign, of the basis by using a bilinear form, integrality and bar invariant properties.

In the last section, we  formulate a new  version of   modified quantum $\mathfrak{osp}(1|2)$ following  ~\cite{BLM90} and ~\cite{Lusztig93}.
As far as we can tell, this is the most natural definition from our presentation of quantum $\mathfrak{osp}(1|2)$ and its geometric construction.
We further observe that our modified quantum $\mathfrak{osp}(1|2)$ is isomorphic to that of Clark-Wang in ~\cite{CW12},
and,  surprisingly, to Lusztig's modified quantum $\mathfrak{sl}(2)$ over a field containing $\mbb Q(v)$ and $t=\sqrt{-1}$.
We arrive at the latter isomorphism by observing the facts that quantum $\mathfrak{osp}(1|2)$ and $\mathfrak{sl}(2)$
have the same $q$-Schur algebras
from the geometric construction and that  modified versions of quantum algebras sit inside the limit of a projective system of  $q$-Schur algebras.
The proof turns out to be extremely easy.
A first consequence of the isomorphism of  modified quantum $\mathfrak{osp}(1|2)$ and $\mathfrak{sl}(2)$ is that there exists
a basis in the modified quantum $\mathfrak{osp}(1|2)$,
coming from the canonical basis of quantum $\mathfrak{sl}(2)$, whose structure constants are in $\mathbb N[v, v^{-1}]$.
Such a  positivity property  in quantum $\mathfrak{osp}(1|2)$  is rather mysterious,
given the fact that  the super sign ``$-1$'' is essentially used in the definition of the quantum $\mathfrak{osp}(1|2)$.
In other words, in modified quantum $\mathfrak{osp}(1|2)$, the super sign ``$-1$'' (or $t^2$ for the modified covering algebra) can be moved outside the structure.
A second consequence of the isomorphism is that  the categories of weight modules of quantum $\mathfrak{osp}(1|2)$ and $\mathfrak{sl}(2)$ are isomorphic to each other.
Moreover, we are able to  construct  very explicit and simple functors of isomorphism  between the two  categories of weight modules.
A third consequence is that Lauda's categorification (\cite{Lauda10}) of quantum $\mathfrak{sl}(2)$
can be served as a version of categorifications of modified quantum $\mathfrak{osp}(1|2)$.

We remark  that the results obtained in this paper can be rephrased in the setting of the  covering algebras (or their modified versions).
We stick to quantum $\mathfrak{osp}(1|2)$ for simplicity.

The coincidence of  modified quantum $\mathfrak{sl}(2)$ and $\mathfrak{osp}(1|2)$ is somehow predicted by various results in literature and, in turn,  explains why the representation theories of the two algebras are identical.
In ~\cite{CFLW}, we will show that  modified quantum algebras and superalgebras (or covering algebras) are isomorphic in general cases.

Meanwhile, Weiqiang Wang  informed us that the equivalence of categories  of weight modules  is known to him more than a year ago using the work ~\cite{Lan02}.
This equivalence   is also proved independently by Kang-Kashiwara-Oh (\cite{KKO13}) in a completely different way and a more general setting.

\subsection{}
We thank Sean Clark and Weiqiang Wang for numerous interesting discussions and their generosities in sharing their ideas with us.
We  thank Alexander Ellis and Aaron Lauda for sending us their slides.
Y. Li is partially supported by the NSF grant DMS-1160351.

\setcounter{tocdepth}{1}
\tableofcontents

\section{Preliminaries}\label{preliminary}
\label{super}

\subsection{}

Let $v$ be an indeterminate and $t$  the imaginary unit such that $t^2=-1$.
For any $k\leq  n\in \mathbb{N}$, we set
$$\begin{array}{lll}
  [n]_v=\frac{v^n-v^{-n}}{v-v^{-1}},& [n]_v^!=\prod_{k=1}^n[k]_v, & \begin{bmatrix}n\\k\end{bmatrix}_v=\frac{[n]_v^!}{[k]_v^![n-k]_v^!},\\
   \begin{bmatrix}n\end{bmatrix}_{v,t}=\frac{(vt)^{n}-(vt^{-1})^{-n}}{vt-(vt^{-1})^{-1}},& [n]_{v,t}^!=\prod_{k=1}^n[k]_{v,t},& \begin{bmatrix}n\\k\end{bmatrix}_{v,t}=\frac{[n]_{v,t}^!}{[k]_{v,t}^!
  [n-k]_{v,t}^!}.
\end{array}$$
One can easily check that
  \begin{equation*}\label{eq22}
\ [n]_{v,t}=t^{n-1}[n]_v,\quad \ [n]^!_{v,t}=t^{\frac{n(n-1)}{2}}[n]^!_v,\quad \ \begin{bmatrix}
      n\\k
    \end{bmatrix}_{v,t}=t^{k
    (n-k)}\begin{bmatrix}
      n\\k
    \end{bmatrix}_v.
  \end{equation*}

 The quantum algebra   $\U$  associated to the ortho-symplectic Lie algebra $\mathfrak{osp}(1|2)$ is, by definition,
 an associative $\mbb{Q}[t](v)$-algebra with 1 generated by the  symbols $E, F, K$ and $K^{-1}$, subject to the following defining relations.
\begin{align}
 KK^{-1}=&1=K^{-1}K. \tag{S1}\\
KE=v^2t^{-2}EK,\ &  \ KF=v^{-2}t^2FK. \tag{S2}\\
EF-t^2FE & =\frac{K-K^{-1}}{v-v^{-1}}.  \tag{S3}
\end{align}
 The above presentation of the algebra $\U$ is new.
Note that the algebra $\U$ is isomorphic to  the algebra  $U_1$ in ~\cite[2.3]{CW12} if the ground field is extended to $\mbb{Q}[t](v)$.
For the  reader's convenience, we provide an isomorphism defined by the following correspondence.
\begin{center}
\renewcommand\arraystretch{1.6}
\begin{tabular}{|c|c|c|c|c|c|c|}
\hline
 $\U$& $E$ &$ F$ & $t^{-1}K$ &$vt^{-1}$&$t^2$ & $[n]_{v,t}$ \\
\hline
$U_1$ &$E$ &$ F$ &$K$&$q$&$\pi$& $[n]^{\pi}_{(vt^{-1})}$\\
\hline
\end{tabular}
\end{center}\vspace{6pt}

The algebra $ \U$  admits a super algebra structure by setting the parity function $p$ to be $p(E)=p(F)=1$ and $p(K)=p(K^{-1})=0$.
By convention, the multiplication on $\U\otimes \U$ is defined by
\begin{equation*}\label{eq71}
(x\otimes y)(x'\otimes y')=t^{2p(y)p(x')}xx' \otimes yy',
\end{equation*}
where $x$,  $y$, $x'$ and $y'$ are homogeneous elements in $\U$.
This gives a super algebra structure on $\U\otimes \U$.
Moreover, a straightforward calculation yields the following proposition.

\begin{prop}\label{prop11}
  There is a unique superalgebra homomorphism
  $\Delta: \U\rightarrow \U\otimes \U$ defined by
  $$\begin{array}{llll}
  & &\Delta(K)=K \otimes K,                     &\Delta(K^{-1})=K^{-1} \otimes K^{-1},\vspace{6pt}\\
  & &\Delta(E)=E\otimes 1+K\otimes E,& \Delta(F)=1 \otimes F+F\otimes K^{-1}.\vspace{6pt}
  \end{array}$$
  \end{prop}
 A standard induction gives rise to the following Lemma.

\begin{lem}\label{lem9}
Let $F^{(n)}=\frac{F^n}{[n]^!_{v,t}}$,  for any $n\in \mbb{N}$.  We have
  $$EF^{(n)}=t^{2n}F^{(n)}E+t^{n-1}F^{(n-1)}\frac{v^{1-n}K-v^{n-1}K^{-1}}{v-v^{-1}}.$$
\end{lem}

By applying Lemma \ref{lem9}, we have the following proposition.

\begin{prop}\label{prop9}
For any $d\in \mbb{N}$, there exist only two non-isomorphic $(d+1)$-dimensional simple highest weight $\U$-modules $\Lambda^{\pm}_d$.
More precisely, the modules $\Lambda^{\pm}_d$ are  spanned by vectors $\xi_0,\xi_1,\cdots, \xi_d$, as vector spaces, and
the action of $\U$ on  $\Lambda^{+}_d$ is given by
\begin{eqnarray}\label{eq21}
F\cdot \xi_r=t^r[r+1]_{v}\xi_{r+1}, \quad E\cdot \xi_r=t^{r-1}[d+1-r]_v\xi_{r-1},\quad K\cdot \xi_r=t^{2r}v^{d-2r}\xi_r,
\end{eqnarray}
while the action  on  $\Lambda^{-}_d$ is  given by
\begin{eqnarray}\label{eq21.1}
F\cdot \xi_r=t^r[r+1]_{v}\xi_{r+1}, \quad E\cdot \xi_r=-t^{r-1}[d+1-r]_v\xi_{r-1},\quad K\cdot \xi_r=-t^{2r}v^{d-2r}\xi_r.
\end{eqnarray}
\end{prop}

 \begin{rem}
   By rewriting the last identity in (\ref{eq21}) into $K\cdot \xi_r=t^dq^{d-2r}\xi_r$, Proposition \ref{prop9} is compatible with the classification of simple modules of
   $U_1\otimes_{\mathbb Q [t] (v)}\mathbb{C} (v)$ in ~\cite[Proposition 3.2]{CW12} and ~\cite[Theorem 3.1]{Zou98}
 \end{rem}

In Sections \ref{sec3} and \ref{sec4}, we will give a geometric categorification of $\Lambda_d^+$.
The module $\Lambda_d^-$ can be categorified similarly.
 {\it From now on, we write  $\Lambda_d$ for $\Lambda_d^+$ for simplicity.}
The module $\Lambda_d$ becomes  a $\mathbb{Z}_2$-graded $\U$-module by setting  the parity function $p$ to be $p(\xi_r)=1$ if $r$ is odd and $p(\xi_r)=0$ otherwise.
For any $\U$-modules $M$ and $N$, the $\U\otimes \U$-module structure on $M\otimes N$ is defined by
\begin{equation*}\label{eq25}
(a\otimes b)\cdot (m\otimes n)=(-1)^{p(b)p(m)}am\otimes bn,
\end{equation*}
for any homogenous element $b\in \U$ and $m\in M$.
Moreover, the $\U$-module structure on $M\otimes N$ is defined by
\begin{equation}\label{eq44}
a \cdot (m\otimes n)=\Delta(a)(m\otimes n), \quad \forall a\in \U, m\otimes n\in M\otimes N.
\end{equation}
For any $\mathbf{d}=(d_1, d_2, \cdots, d_m)\in \mathbb{N}^m$, let
\begin{equation}
\label{Ld}
\Lambda_{\mathbf{d}}=\Lambda_{d_1}\otimes \Lambda_{d_2}\otimes \cdots \otimes \Lambda_{d_m}.
\end{equation}
Let
\[
\A =\mbb Z[v^{\pm 1}, t]
\]
denote the subring in $\Q(v)[t]$ of Laurent polynomials. We denote by $_{\A}\!\U$ the $\A$-subalgebra of $\U$ generated by $E^{(n)}:=\frac{E^n}{[n]^!_{v,t}}$, $F^{(n)}$ and $K^{\pm 1}$ for $n\in \mbb N$. It is clear that the comultiplication $\Delta$ induces a comultiplication on $_{\A}\!\U$. 
By combining with (\ref{eq21}), we can define  the integral form $_{\A}\!\Lambda_{\bf d}$ of the module $\Lambda_{\mbf d}$.

\subsection{} 
\label{perversesheaf}
In this section, we recall notations and facts in  the theory of mixed perverse sheaves.
We refer to ~\cite[Chapter 8]{Lusztig93} and \cite{BBD82} for more details.

Let $k$ be an algebraically closed field of positive characteristic.
Let $l$ be a fixed prime number invertible in $k$, and $\bar{\mathbb{Q}}_l$ be the algebraic closure of the field of $l$-adic numbers.
Denote by $\mathcal{D}(X)=\mathcal{D}^b_c(X)$ the bounded derived category of
$\bar{\mathbb{Q}}_l$-constructible sheaves on  the algebraic variety $X$ over $k$.
Let $\mathcal{D}_m(X)$ be the full subcategory of $\mathcal{D}(X)$ consisting of all mixed complexes.
If two complexes $K$ and $L$  in $\D(X)$ are isomorphic, we write $K=L$.

Let ${\rm wt}(L)$ denote the weight of a pure complex $L$.
Let $[-]$, $(-)$, $\mbb D$ and $\otimes$ denote the shift functor, Tate twist functor, Verdier duality functor and tensor product functor, respectively.
To a morphism
$f: X \rightarrow Y$ of algebraic varieties over $k$,
we can associate  four of Grothendieck's six operations $f^*, f^!: \mathcal{D}(Y) \rightarrow \mathcal{D}(X)$ and $f_*, f_!: \mathcal{D}(X) \rightarrow \mathcal{D}(Y)$.
We recall some facts  which we will  use freely later.
\begin{itemize}
\item[(1)] Simple perverse sheaves are pure.

\item[(2)] Functors $(-)$ and $[-]$ commute with each other and with  all functors $f^*,f_!,f^!,f_*$.

\item[(3)] $\mbb{D} f_!=f_* \mbb{D},\quad \mbb{D} f^!=f^* \mbb{D},\quad \mbb{D}(L[n](m))=(\mbb{D}L)[-n](-m).$

\item[(4)] ${\rm wt}(L[n])={\rm wt}(L)+n,\ {\rm wt}(L(n))={\rm wt}(L)-2n, \ {\rm wt}(\mbb{D}L)=-{\rm wt}(L)$ for any pure complex $L$.

\item[(5)] If $f: X \rightarrow Y$ is smooth with connected fibers of equal dimension, then ${\rm wt}(f^*L)={\rm wt}(L)$ for any pure complex $L$.

\item[(6)]If $f: X \rightarrow Y$ is a proper morphism, then ${\rm wt}(f_!L)={\rm wt}(L)$ for any pure complex $L$.

\item[(7)]
$f_!(L\otimes f^* M) = f_!L\otimes M$,\quad $f^*(L\otimes M) = f^*L\otimes f^* M$.

\item[(8)] If the following square  is  cartesian and $f$ is a proper map
\[\xymatrix{ Z \ar[r]^{g'} \ar[d]_{f'} & X \ar[d]^{f}\\
X' \ar[r]^{g} & Y,}\]
then we have $g^*f_! L = f'_!g'^*L$ for any complex $L\in \D(X)$.

\end{itemize}

\subsection{}
\label{convolution}

Suppose that $X_1, X_2 $ and $X_3$ are three algebraic varieties over $k$.
Let $p_{ij}: X_1\times X_2 \times X_3 \to X_i \times X_j$ be the projection to the $(i, j)$-factor, for $(i, j)=(1, 2), (2,3), (1,3)$.
For any $L\in \mathcal D(X_1\times X_2)$ and $M\in \mathcal D(X_2\times X_3)$, we set
\begin{equation}
\label{convolution-a}
L\circ M=(p_{13})_! ( p_{12}^* L \otimes p_{23}^* M) \quad  \in \mathcal D(X_1\times X_3).
\end{equation}

\begin{lem}
\label{convolution-associative}
Assume the above set up, we have $(L\circ M) \circ N = L\circ (M\circ N)$, where
$N$ is  any complex in $\mathcal D(X_3\times X_4)$ and
$X_4$ is a fourth variety over $k$.
\end{lem}

The proof is standard and  left to the reader.

\section{A geometric categorification of $\U$}
\label{sec3}

\subsection{}

We fix a positive  integer  $d$.
We  write $\F_r$ for the Grassmannian of all dimension $r$ subspaces in $k^d$.
It is clear that $\F_r$ is empty unless $r$ subjects to $0\leq r\leq d$.
The group $G=GL(k^d)$ naturally acts on $\F_r$ from the left.
Let $G$ act on  $\F_{r}\times \F_{r'}$ diagonally.
By ~\cite[Section 1]{BLM90}, the $G$-orbits in $\F_r\times \F_{r'}$ are parametrized by
the set $\Theta_d(r, r')$ of all $2\times 2$ matrices $(a_{ij})$ such that $a_{11}+a_{12}=r$, $a_{11}+a_{21}=r'$ and $\sum_{i,j=1, 2} a_{ij}=d$.
In fact,  a bijection is defined by sending   a pair $(V, V')\in \F_r\times \F_{r'}$ to  the following  matrix.
\[
\left [
  \begin{array}{cc}
    |V\cap V'| & |V/V\cap V'|\\
   |V+V'/V| &  d-|V+V'|
  \end{array}
 \right ],\quad
\mbox{where $|V|=\dim V$.}
\]
We set
\begin{equation*}\label{eq35}
  \Theta_d=\sqcup_{r, r'} \Theta_d(r, r').
\end{equation*}
We define the following closed subvarieties in $\F_r\times \F_{r'}$ for appropriate  $r$ and $ r'$.
\begin{equation}
\label{F-a}
\begin{split}
\F_{r, r+a} &= \{ (V, V') \in \F_r \times \F_{r+a}| V\subset V' \}, \quad  \forall 0 \leq r, a, r+a \leq d;\\
\F_{r, r-a}  &= \{ (V, V') \in \F_r \times \F_{r-a}| V\supset V' \}, \quad \forall 0 \leq r, a , r-a \leq d.
\end{split}
\end{equation}
We denote
\begin{equation*}
\begin{split}
E_{r, r+a}&=  (\bar{\mbb Q}_l)_{\F_{r, r+a}}   [ a(d-(r+a)) ] \left (\frac{a(d-a)}{2} \right ) \quad  \in \mathcal D(\F_r\times \F_{r+a}),\\
F_{r, r-a}&=  (\bar{\mbb Q}_l)_{\F_{r, r-a}} [ a(r-a)](a(r-a))        \quad  \in \mathcal D(\F_r\times \F_{r-a}),\\
1_r&=  (\bar{ \mbb Q}_l)_{\F_{r, r}}                    \quad  \in \mathcal D(\F_r\times \F_r),
\end{split}
\end{equation*}
where $ a>0$ and  $(\bar{\mbb Q}_l)_{X_1}  \in \mathcal D(X)$ denotes the extension by zero of the constant sheaf
$(\bar{\mbb Q}_l)_{X_1} $  in $\mathcal D(X_1)$ for a given subvariety $X_1$ in $X$.
If the variety $\F_{r, r'}$ is empty, the associated complex is defined to be zero.

\begin{lem}
\label{sl-E}
$(a)$.  $\; E_{r, r+a}  E_{r+a, r+a+1}= \bigoplus_{j=0}^a E_{r, r+a+1} [a-2j](a-j)$.

$(b)$. $\; F_{r, r-a}  F_{r-a, r-a-1} = \bigoplus_{j=0}^a F_{r, r-a-1} [a-2j](a-j)$.
\end{lem}

\begin{proof}
The support of the complex
$p_{12}^* E_{r, r+a} \otimes p_{23}^* E_{r+a, r+a+1}$ is
\begin{equation*}
\label{eq5}
S=\{ (V, V', V'')\in \F_r \times \F_{r+a} \times \F_{r+a+1} | V\subset V' \subset V'' \}.
\end{equation*}
The restriction of $p_{13}$ to $S$ is a $\mbb P^a$-bundle. So we have
\begin{equation*}
\begin{split}
E_{r, r+a}&  E_{r+a, r+a+1}
=(p_{13})_! (p_{12}^* E_{r, r+a} \otimes p_{23}^* E_{r+a, r+a+1})\\
&=(p_{13})_! ( \bar{\mbb Q}_l)_S   [ a (d-r-a) + d-(r+a)-1]\left (\frac{a(d-a)}{2}+\frac{d-1}{2}\right )\\
&=\bigoplus_{j=0}^a (\bar{\mbb Q}_l)_{\F_{r, r+a+1}} [-2j](-j) [(a+1)(d-(r+a)) -1]\left(\frac{(a+1)(d-1-a)}{2}+a\right)\\
&=\bigoplus_{j=0}^a E_{r, r+a+1} [a-2j](a-j).
\end{split}
\end{equation*}

Similarly, the support of
$p_{12}^* F_{r, r-a} \otimes p_{23}^* F_{r-a, r-a-1}$
is
\[
S=\{ (V, V', V'')\in \F_r \times \F_{r-a} \times \F_{r-a-1} | V\supset V' \supset V'' \}.
\]
The restriction of $p_{13}$ to $S$ is again a $\mbb P^a$-bundle.
By a  similar argument as above, we have the second identity.
\end{proof}

\begin{lem}
\label{sl-relation}
$(a). \quad 1_r  1_{r'}  = \delta_{r, r'} 1_r$.
\begin{equation*}
\begin{split}
(b). \quad& E_{r, r+1}   1_{r'} =  \delta_{r+1, r'} E_{r, r+1}, \quad \ \quad
1_{r'}   E_{r, r+1} = \delta_{r', r} E_{r, r+1}.\\
 ( c). \quad& F_{r, r-1}  1_{r'}=  \delta_{r-1, r'} F_{r, r-1}, \quad \ \quad
1_{r'}  F_{r, r-1}=  \delta_{r', r} F_{r, r-1}.\\
(d). \quad& E_{r, r+1}  F_{r+1, r} \oplus \textstyle{\bigoplus_{0\leq j < 2r-d}}   1_r  [2r-2j-1-d] \left (2r-j-\frac{d+1}{2} \right )  \\
&\qquad= F_{r, r-1}  E_{r-1, r}(1) \oplus  \textstyle{\bigoplus_{ 0\leq j< d-2r}} 1_{r} [d -2j-1 -2r]\left (\frac{d-1}{2}-j \right ).\hspace{45pt}
\end{split}
\end{equation*}
\end{lem}

\begin{proof}
Let us show that $1_r 1_{r'}=  \delta_{r, r'} 1_r$.  Assume that $r=r'$. As before, let
$p_{ij}: \F_r\times \F_r \times \F_r \to \F_r\times \F_r $ be the projection to the $(i, j)$-factor.
By the definition of $1_r$, the support of $p_{12}^* 1_r \otimes p_{23}^* 1_r$ is the variety
$S=\{(V, V', V'') \in \F_r\times \F_r \times \F_r | V=V'=V''\}$, i.e., the diagonal of the variety $\F_r\times \F_r\times \F_r$.
So the image of $S$ under $p_{ij}$ is exactly $\F_{r,r}$ for $(i, j)=(1,2), (2,3), (1,3)$.  Moreover, the restriction of $p_{ij}$ to $S$ is an isomorphism.
Thus the restriction of $1_r 1_r$ to $\F_{r, r}$ is the constant sheaf.  Therefore, we have $1_r 1_r= 1_r$.
If $r\neq r'$, then the support of $p_{12}^* 1_r\otimes p_{23}^* 1_r$ is empty. So we have $1_r 1_{r'} = 0$.

Next, let us show that $E_{r, r+1} 1_{r'} =  \delta_{r+1, r'} E_{r, r+1}$. Assume that $r'=r+1$.
The support of $p_{12}^* E_{r, r+1} \otimes p_{23}^* 1_r$ is
$S=\{ (V, V', V'')\in \F_r \times\F_{r+1} \times \F_{r+1} | V\subset V', V'=V''\}$.
By definition, the restriction of $p_{12}^* E_{r, r+1} \otimes p_{23}^* 1_r$  to $S$ is
$(\bar{\mbb Q}_l)_S [ d-(r+1)](\frac{d-1}{2})$. Note that the restriction of $p_{13}$ to $S$ is again an isomorphism, and the image of $p_{13} $ is $\F_{r, r+1}$. Therefore we have
$E_{r, r+1} 1_{r+1}=E_{r, r+1}$. For the case of $r'\neq r$, the identity holds by definitions. One may show similar identities in the lemma in a similar way.

Finally, let us show that the last identity in the lemma.
Let us compute the complex $E_{r, r+1} F_{r+1, r}$.
The support of $p_{12}^* E_{r, r+1} \otimes p_{23}^* F_{r+1, r}$ is
\[
S=\{ (V, V', V'')\in \F_r\times \F_{r+1} \times \F_{r}| V \subset V' \supset V''\}.
\]
Let $S_1=\{ (V, V', V'') \in S| V=V''\}$ and $S_2=S \backslash S_1$. Then
\[S_2\simeq S_2'\overset{\mrm{def.}}{=}\{(V, V'')\in \F_r \times \F_r|  |(V+V'')/V|=1= |(V+V'')/V''|\}.\]
Observe that the restriction $p_{13}'$ of $p_{13}$ to $S_1$ is a fiber bundle of fiber isomorphic to  the projective space $\mbb P^{d-r-1}$,
while the restriction $p_{13}''$ of $p_{13}$ to $S_2$ is an isomorphism. Further, the image of $p_{13}'$ is $\F_{r, r}$,
and the image of $p_{13}''$ is $S_2'$.
The restriction of $p_{12}^* E_{r, r+1} \otimes p_{23}^* F_{r+1, r}$ to $S$ is
\[
(\bar{\mbb Q}_l)_S [ d-(r+1) ](\frac{d-1}{2})\otimes (\bar{\mbb Q}_l)_S [r](r) =(\bar{\mbb Q}_l)_S  [d-1]\left(\frac{d-1}{2}+r\right).
\]
So
\begin{equation*}
\begin{split}
E_{r, r+1} F_{r+1, r}
&=(p_{13})_! (p_{12}^* E_{r, r+1} \otimes p_{23}^* F_{r+1, r})
=(p_{13})_! ( \bar{\mbb Q}_l)_S  [d-1]\left(\frac{d-1}{2}+r\right)\\
&=(p_{13}')_! ( \bar{\mbb Q}_l)_{S_1}[d-1](\frac{d-1}{2}+r)\oplus (p_{13}'')_! ( \bar{\mbb Q}_l)_{S_2}[d-1]\left(\frac{d-1}{2}+r\right)\\
&=\large ( \oplus_{j=0}^{d-r-1} (\bar{\mbb Q}_l)_{\F_{r,r}} [-2j](-j) \oplus (\bar{\mbb Q}_l)_{S_2'} \large ) [d-1]\left(\frac{d-1}{2}+r\right),
\end{split}
\end{equation*}
where the third equation is due to ~\cite[8.1.6]{Lusztig93}.

Similarly, we compute $F_{r, r-1}  E_{r-1, r}$ and get
\[
F_{r, r-1}  E_{r-1, r}(1)=
(\oplus_{j=0}^{r-1} (\bar{\mbb Q}_l)_{\F_{r, r}} [-2j](-j)\oplus (\bar{\mbb Q}_l)_{S_2''} ) [d-1]\left(\frac{d-1}{2}+r\right),
\]
where $S_2''=\{ (V, V'')\in \F_r\times \F_r | |V/(V\cap V'')|=1=|V''/ (V\cap V'')| \}$.
Observe that $S_2'=S_2''$ and
\[
(\oplus_{j=0}^{d-r-1} 1_r     \oplus \oplus_{d-r\leq j < r} 1_r) [-2j](-j)
=(\oplus_{j=0}^{r-1} 1_r  \oplus
\oplus_{ r \leq j < d-r} 1_r ) [-2j](-j).
\]
We have the last identity in the lemma.
The lemma follows.
\end{proof}

For any  $A=(a_{ij})_{1\leq i, j \leq 2}$, we set
\[
\{A\}= IC(O_A)[-r(A)](-r(A)/2), \quad \forall A\in \Theta_d,
\]
where $O_A$ is the corresponding $G$-orbit of $A$, $IC(O_A)$ is the intersection cohomology complex attached to the closure of $O_A$ (\cite{BBD82}), and $r(A)=(a_{11}+a_{12})(a_{21}+a_{22})$ is the dimension of the image of $O_A$ under the first projection.

\begin{lem}
\label{sl-simple}
$(a). \quad
E_{r-a, r} 1_r F_{r, r-b}(n_1)=
 \left \{
  \begin{array}{cc}
    r-a-b & b\\
   a  &  d-r
  \end{array}
 \right\},
 \quad \forall d \leq (r-a)+ (r-b).$
\begin{align*}
\begin{split}
(b). &\quad
 F_{r+b, r} 1_r E_{r, r+a}(n_2)=
 \left \{
  \begin{array}{cc}
    r & b\\
   a  &  d-r-a-b
  \end{array}
 \right\}, \quad \forall d \geq (r+a)+ (r+b).\\
 (c).&\quad
E_{r-a, r} 1_r F_{r, r-b}=F_{d-r+b, d-r} 1_{d-r} E_{d-r, d-r+a}(ab), \quad \quad  \mbox{if}\ \; d=(r-a)+(r-b),
\end{split}
\end{align*}
where $n_1=-\frac{1}{2}(a(r-a)+b(r-b))$ and $n_2=-\frac{1}{2}(ar+br)$.
\end{lem}

\begin{proof}
We  prove the first equation.
The support of the complex $p_{12}^* E_{r-a, r}\otimes p_{23}^* F_{r, r-b}$ is
\[
S=\{ (V, V'', V')\in \F_{r-a} \times \F_r \times \F_{r-b} |  V \subset V'' \supset V'\}.
\]
By definition, we have
\begin{equation}
E_{r-a, r} F_{r, r-b}=(p_{13})_! (\bar{\mbb Q}_l)_S   [ a(d-r)+ b(r-b)]\left (\frac{a(d-a)}{2}+b(r-b) \right ).
\label{sl-simple-A}
\end{equation}
Consider the restriction of $p_{13}$ to $S$.
The image of $S$ under $p_{13}$ consists of  the pairs $(V, V')\in \F_{r-a}\times \F_{r-b}$ such that
$|V+ V'| \leq r$.
Thus we have
\[
r-a-b \leq |V\cap V'| \leq \mrm{min} \{ r-a, r-b\}.
\]
Recall from ~\cite[2.3]{BLM90} that
\begin{equation}
\label{sl-simple-dimension}
\begin{split}
d&(A)-r(A)
=a_{11}a_{12}+a_{21}a_{12}+a_{21} a_{22}\\
&=(r-a-b -|V\cap V'|) (d-r+b+|V\cap V'| ) + b |V\cap V'|  + a(d-r+b).
\end{split}
\end{equation}
In particular,
\begin{equation}
\label{sl-maximaldim}
d(A)-r(A)=a(d-r)+b(r-b), \quad \mrm{if} \quad |V\cap V'|=r-a-b.
\end{equation}
From (\ref{sl-simple-dimension}), we see that
$p_{13}(S)$ is the orbit closure of the $G$-orbit  whose associated matrix is
\[
A_0=
 \left [
  \begin{array}{cc}
    r-a-b & b\\
   a  &  d-r
  \end{array}
 \right].
\]
We claim that
\begin{equation}
\mbox{
The restriction of $p_{13}$ to $S$ is a small resolution.}
\label{sl-simple-B}
\end{equation}
Recall that smallness means that the following two conditions are satisfied.
\begin{itemize}
\item[(a).]  $ 2 |p_{13}^{-1} (x)| \leq  d(A_0)- d(A)$, for any $x\in O_A \subseteq O_{A_0}$.
\item[(b).] The equality holds if and only if $A=A_0$.
\end{itemize}
We show that $p_{13}$ satisfies (a).
Given any pair $(V, V')$ in $p_{13}(S)$, the dimension of the fiber $p_{13}^{-1} (V, V')$ is
\begin{equation}
\label{sl-simple-fiber}
| p_{13}^{-1} (V, V') |= | Gr(d- |V+V'|, r-|V+V'|)  |= (r-|V+V'|) (d-r).
\end{equation}
Since $r(A_0)=r(A)$, we have
 \[
 2 |p_{13}^{-1} (x)| -( d(A_0)- d(A))=
 2 |p_{13}^{-1} (x)| -( d(A_0)-r(A_0))- ( d(A)-r(A)).
 \]
By (\ref{sl-simple-dimension}) and (\ref{sl-simple-fiber}), we have
\begin{equation}
\label{sl-simple-semismall}
2 |p_{13}^{-1} (x)| -( d(A_0)- d(A))=
(r-a-b -|V\cap V'| ) (-d+r +|V\cap V'|) \leq 0.
\end{equation}
This shows (a).
The inequality (\ref{sl-simple-semismall})  becomes equality if and only if $|V\cap V'| =r-a-b$. So (b) holds for $p_{13}$.
It is clear that the restriction of $p_{13}$ to $p_{13}^{-1} (O_{A_0})$ is an isomorphism. The claim follows.

By  (\ref{sl-simple-A}) , (\ref{sl-maximaldim})  and (\ref{sl-simple-B}),  we have
$E_{r-a, r} F_{r, r-b}= \{A_0\}$ up to a Tate twist. Since ${\rm wt}(\{A_0\})=0$, we have $n_1=-\frac{1}{2}(a(r-a)+b(r-b))$
by checking  the weight of $E_{r-a, r} F_{r, r-b}$.
The first equation in the lemma follows.
The second equation can be shown similarly. The third one follows from the first two equations.
\end{proof}

\subsection{}
\label{sec3.4}
Consider the following complex
\begin{align}
\label{complex}
L_1 \circ L_2 \circ \cdots \circ L_m, \quad m\in \mbb N,
\end{align}
where the  $L_i$'s   are either  $E_{r, r+1}$,  $F_{r, r-1}$, or  $1_r$.
Assume that $L_i\in \mathcal D(\F_{r_i}\times \F_{r_{i+1}})$ for $i=1,\cdots, m$.
Let $s_{ij}: \prod_{k=1}^{m+1} \F_{r_k} \to \F_{r_i}\times \F_{r_j}$ be the projection to $(i,j)$-factor.
By applying \ref{perversesheaf} (7) and (8), we get
\[
L_1 \circ L_2 \circ \cdots \circ L_m=(s_{1,m+1})_! (\otimes_{i=1}^m s_{i, i+1}^* (L_i)).
\]
Observe that $s_{1, m+1}$ is proper, and the restriction of the complex $\otimes_{i=1}^m s_{i, i+1}^* (L_i)$
to its support, which is smooth and irreducible, is a constant sheaf with a shift and a Tate twist. By the decomposition theorem (\cite{BBD82}), we see that the complex (\ref{complex}) is  semisimple.

\subsection{}
\label{sec3.5}

Let $\mathcal Q_d^{r, r'}$ be the full subcategory of $\mathcal D(\F_r\times \F_{r'})$ consisting of   semisimple complexes, whose simple constitutes are direct summands of
the  complex  (\ref{complex}) up to shifts and twists.

Let $\mathbf{Q}_d^{r,r'}$ be the split Grothendieck group of $\mathcal{Q}_d^{r,r'}$.
More precisely, $\mathbf{Q}_d^{r,r'}$ is the abelian group generated by the isomorphism classes of  objects in $\mathcal{Q}_d^{r,r'}$ and  subject to the following relation.
\begin{equation}\label{eq19}
  \langle C\oplus C'\rangle=\langle C\rangle+\langle C'\rangle, \quad \forall C,C'\in \mathcal{Q}_d^{r,r'}.
\end{equation}
Let $\mathbf{Q}_d=\oplus_{r,r'\in \mbb Z_{\geq 0}}\mathbf{Q}_d^{r,r'}$
and $\tilde \A=\mbb{Z}[v^{\pm 1},  \tau^{\pm 1}]$,  where $ \tau$ is  an indeterminate.
We define an $\tilde \A$-module structure on $\mathbf{Q}_d$ as follows.
\begin{equation}\label{eq7}
  v \cdot \langle C\rangle=\langle C[1](\frac{1}{2})\rangle, \quad  \tau\cdot\langle C\rangle=\langle C(\frac{1}{2})\rangle, \quad \forall \langle C\rangle\in \mathbf{Q}_d.
\end{equation}
By the property of the shift and Tate twist functors, this action is well-defined.
Recall that
\[
\A =\mathbb Z[v^{\pm 1}, t].
\]
There is an obvious ring homomorphism $\tilde \A \to \A$ by sending $\tau $ to $t$.
Let
\[{}_{\mathcal{A}}\! \mbf S_{v, t}(2,d)=
\A \otimes_{\tilde \A}
\mathbf{Q}_d.
\]
By the $ \tau$-action in (\ref{eq7}), we have
\begin{equation}\label{eq4}
1\otimes  \langle C(2)\rangle  =1\otimes  \tau^4 \langle C\rangle = t^4\otimes  \langle C\rangle =
 1\otimes \langle C\rangle,\quad  \forall \langle C\rangle  \in \mbf Q_d.
\end{equation}
 Let
\begin{equation}
\label{S}
\mbf S_{v,t}(2,d)=\mbb{Q}[t](v) \otimes_{\mathcal{A}} {}_{\mathcal{A}}\! \mbf S_{v, t}(2,d).
\end{equation}

By (\ref{eq19}), (\ref{eq4}) and Lemma ~\ref{sl-simple},  the convolution product $``\circ"$ in (\ref{convolution-a})
descends to  a bilinear map on  ${}_{\mathcal{A}}\! \mbf S_{v, t}(2,d)$:
$$  \circ: \, {}_{\mathcal{A}}\! \mbf S_{v, t}(2,d) \times \, {}_{\mathcal{A}}\! \mbf S_{v, t}(2,d)\to \, {}_{\mathcal{A}}\!\mbf S_{v, t}(2,d).$$
It is associative due to Lemma ~\ref{convolution-associative}.
Together with ``$\circ$'', the space ${}_{\mathcal{A}}\!\mbf S_{v, t}(2,d)$ becomes an associative algebra over $\mathcal A$.
By an abuse of notation, we  write $C$ instead of $1\otimes \langle C\rangle$ for elements in ${}_{\mathcal{A}}\! \mbf S_{v, t}(2,d)$.

\begin{lem}
\label{lem1}
The following identities hold in ${}_{\mathcal{A}}\! \mbf S_{v, t}(2,d)$.

$(a)$.  $\; E_{r, r+a}  E_{r+a, r+a+1}= [a+1]_{v, t} E_{r, r+a+1}$.

$(b)$. $\; F_{r, r-a}  F_{r-a, r-a-1}= [a+1]_{v, t}F_{r, r-a-1}$.
\end{lem}

\begin{proof}
  By Lemma \ref{sl-E} and $\A$-action on ${}_{\mathcal{A}}\! \mbf S_{v, t}(2,d)$ defined above, we have
  $$ E_{r, r+a} \circ E_{r+a, r+a+1}=\sum_{j=0}^av^{a-2j} t^{a} E_{r, r+a+1}=[a+1]_{v, t} E_{r, r+a+1}.$$
  The second identity can be proved similarly.
\end{proof}

Let $K_r=1_r[d-2r](\frac{d}{2})$, $K_r^{-1}=1_r[2r-d](-\frac{d}{2})$ and
\begin{eqnarray*}\label{eq24}
  E=\sum_{r=0}^{d-1} E_{r,r+1}, \quad F=\sum_{r= 1}^d F_{r,r-1},\quad K=\sum_{r=0}^dK_r,\quad
  K^{-1}=\sum_{r=0}^dK_r^{-1}.
\end{eqnarray*}
By Lemmas \ref{sl-relation}, \ref{sl-simple}, \ref{lem1} and using $t^2=-1$, we have

\begin{thm}
\label{prop15}
There exists a unique surjective algebra homomorphism $\chi: \U \rightarrow \mbf S_{v,t}(2,d)$ by sending the generators in $\U$ to the respective elements in $\mbf S_{v, t}(2, d)$.
Moreover, it induces a surjective  $\A$-algebra homomorphism  from the integral form $_{\A}\!\U$ of $\U$ to the algebra $_{\A}\! \mbf S_{v, t}(2,d)$.
\end{thm}

\subsection{}
\label{S_v}
Let $\mbf S_v(2,d)$ be the $q$-Schur algebra associated to $\mathfrak{sl}(2)$. By \cite{BLM90} and ~\cite[1.3]{D95}, $\mbf S_{v,t}(2,d)$ is isomorphic to
$\mbb{Q}[t](v) \otimes_{\mbb{Q}(v)} \mbf S_v(2,d)$.
We define a $\mbb{Q}[t](v)$-linear map
\begin{equation}
\label{psi}
\psi_{d,d+2}:  \mbf S_{v,t}(2,d+2) \rightarrow \mbf S_{v,t}(2,d)
\end{equation}
 by
\begin{equation}
\label{eq33}
\{A\}\mapsto
\left\{
\begin{array}{ll}
  \{A-I_{2\times 2}\},& \quad {\rm if}\ A-I_{2\times 2}\in \Theta_d,\vspace{6pt}\\
  0,&\quad {\rm otherwise},
\end{array}
\right.
\end{equation}
where $I_{2\times 2}$ is the  identity matrix of rank 2.

\begin{prop}\label{prop13}
  For any $d\in \mbb{Z}_{> 0}$, $\psi_{d,d+2}$ in (\ref{psi})  is a surjective algebra homomorphism.
\end{prop}

\begin{proof}
By  Lemma \ref{sl-simple}, $\mbf S_{v,t}(2,d)$ is generated by $1_r, E_{r,r+1}$ and $F_{r,r-1}$, $\forall 0\leq r\leq d$
and subjects to the relations given by Lemma \ref{sl-relation}. 
 By (\ref{eq33}), we have
$$\psi_{d,d+2}(E_{r,r+1})=tE_{r-1,r},\quad \psi_{d,d+2}(F_{r,r-1})=tF_{r-1,r-2}, \quad{\rm and}\quad \psi_{d,d+2}(1_r)=1_{r-1}.$$
$\psi_{d,d+2}$ is surjective, since    $1_r, E_{r,r+1}$ and $F_{r,r-1}$ are algebraic generators of $\mbf S_{v,t}(2,d)$.
The rest is to show  that $\psi_{d,d+2}$ is compatible with the defining relations of $\mbf S_{v,t}(2,d+2)$.
 We only check the relation $E_{r,r+1} F_{r+1,r}-t^2F_{r,r-1}  E_{r-1,r}=t^{2r}[d +2-2r]_v1_r$ and the rest relations can be checked similarly. By the definition of $K_r$, we have
$\psi_{d,d+2}(K_r)=v^{d+2-2r}t^{2r}\psi_{d,d+2}(1_r)=v^{d+2-2r}t^{2r}1_{r-1}=t^2K_{r-1}$. So
$$\psi_{d,d+2}(E_{r,r+1} F_{r+1,r}-t^2F_{r,r-1}  E_{r-1,r})=\psi_{d,d+2}(\frac{K_r-K_r^{-1}}{v-v^{-1}})=t^2\frac{K_{r-1}-K_{r-1}^{-1}}{v-v^{-1}}.$$
On the other hand, we have
\begin{eqnarray*}
& &\psi_{d,d+2}(E_{r,r+1}) \psi_{d,d+2}(F_{r+1,r})-t^2\psi_{d,d+2}(F_{r,r-1}) \psi_{d,d+2} (E_{r-1,r})\\
&=&t^2(E_{r-1,r} F_{r,r-1}-t^2F_{r-1,r-2}  E_{r-2,r-1})=t^2\frac{K_{r-1}-K_{r-1}^{-1}}{v-v^{-1}}.
\end{eqnarray*}
The proposition follows.
\end{proof}

\section{A geometric categorification of $\U$-modules}
\label{sec4}

\subsection{}

Let $\mathcal{C}_{r, r'}$ be the category of triangulated functors from $\D(\F_{r'})$ to $\D(\F_r)$. Consider the following diagram
\[\xymatrix{ \F_r &\F_{r}\times \F_{r'} \ar[l]_-{p_1} \ar[r]^-{p_2} & \F_{r'},
}\]
where $p_1$ and $p_2$ are projections to the first and second components, respectively.
Define a functor
\[\Psi_{r,r'}: \D(\F_r \times \F_{r'}) \rightarrow \mathcal{C}_{r,r'}\]
by $\Psi_{r,r'}(L)=p_{1!}(L \otimes p_2^*(-))$ for any object $L$ in $ \D(\F_r \times \F_{r'})$.

\begin{prop}\label{prop4}
  $\Psi_{r'',r}(L \circ M)=\Psi_{r'',r'}(L)\Psi_{r',r}(M)$ for any objects $L$ in $\D(\F_{r''} \times \F_{r'})$ and $M$ in $\D(\F_{r'} \times \F_r)$.
\end{prop}

This is a special case of Proposition 7.2 in \cite{Li10}.

Consider the following diagram
\[\xymatrix{\F_r &\F_{r,r+a} \ar[l]_{p} \ar[r]^{p'} & \F_{r+a},
}\]
where $\F_{r, r+a}$ is defined in (\ref{F-a}) and $p$, $p'$ are projections.
For any $0 \leq r,  r+a \leq d$ and $a>0$, we define
\begin{equation}\label{eq31}
  \begin{split}
    \mrk{K}_r&=Id[d-2r] \left (\frac{d}{2} \right ):\D(\F_r) \rightarrow \D(\F_r),\\
    \mrk{E}_{r, r+a}&=p_!p'^*[a(d-a-r)]\left (\frac{a(d-a)}{2} \right ): \D(\F_{r+a}) \rightarrow \D(\F_{r}),\\
     \mrk{F}_{r+a,r}&=p'_!p^*[ar](ar): \D(\F_{r}) \rightarrow \D(\F_{r+a}).
  \end{split}
\end{equation}

\begin{lem}\label{lem4}
  For any $0 \leq r,  r+a \leq d$ and $a>0$, we have
  \begin{eqnarray*}
\Psi_{r,r+a}(E_{r,r+a})=\mrk{E}_{r,r+a},\;
\Psi_{r,r-a}(F_{r,r-a})=\mrk{F}_{r,r-a}, \;
{\rm and}\ \Psi_{r, r} (K_r ) = \mrk K_r.
  \end{eqnarray*}
\end{lem}

\begin{proof}
We show that $ \Psi_{r,r+a}(E_{r,r+a})=\mrk{E}_{r,r+a}$. The other identities can be proved similarly.
We notice that the shift and Tate twist are the same in the complex $E_{r,r+a}$ and in the functor $\mrk{E}_{r,r+a}$, respectively.
So it is enough to show that $p_!p'^*(C)=p_{1!}(\iota_!(\bar{\mathbb{Q}}_l)_{\F_{r,r+a}} \otimes p_2^*(C))$ for any object $C$ in $\D(\F_{r+a})$, where $p,p', p_1$ and $p_2$ are the obvious maps in the following commutative diagram and $\iota$ is the closed embedding.
\[\xymatrix{\F_r & \F_{r,r+a} \ar[l]_{p} \ar[r]^{p'} \ar[d]^{\iota}  & \F_{r+a} \\
& \F_r \times \F_{r+a}. \ar[lu]^{p_1} \ar[ru]_{p_2}&
}\]
By the projection formula \ref{perversesheaf} (7) and the commutativity of the diagram, we have
\begin{eqnarray*}
  p_{1!}(\iota_!(\bar{\mathbb{Q}}_l)_{\F_{r,r+a}} \otimes p_2^*(C))=p_{1!}\iota_!((\bar{\mathbb{Q}}_l)_{\F_{r,r+a}} \otimes \iota^*p_2^*(C))
  =p_{1!}\iota_!\iota^*p_2^*(C))=p_!p'^*(C).
\end{eqnarray*}
The lemma follows.
\end{proof}

By Proposition  \ref{prop4} and Lemma \ref{lem4}, we can transport results on the complexes $E$, $F$ and $K$ to the corresponding  functors $\mrk E$, $\mrk F$, and $\mrk K$.
In particular, we have

\begin{lem}
\label{lem3}
$(a)$.  $\;\mrk{E}_{r, r+a}\mrk{E}_{r+a, r+a+1} = \bigoplus_{j=0}^a \mrk{E}_{r, r+a+1} [a-2j](a-j)$.
\begin{equation*}
  \begin{split}
(b). \; \mrk{F}_{r, r-a}\mrk{F}_{r-a, r-a-1}&= \oplus_{j=0}^a \mrk{F}_{r, r-a-1} [a-2j](a-j).\\
(c).\; \mrk{E}_{r, r+1}\mrk{F}_{r+1, r} \oplus \oplus&_{0\leq j < 2r-d}{\rm Id}\  [2r-2j-1-d] \left (2r-j-\frac{d+1}{2} \right )  = \hspace{43pt}\\
\mrk{F}_{r, r-1}&\mrk{E}_{r-1, r} (1) \oplus \oplus_{ 0\leq j< d-2r} {\rm Id}\  [d -2j-1 -2r] \left (\frac{d-1}{2}-j \right ).
\end{split}
\end{equation*}
\end{lem}

Let $$\mrk{K}=\oplus_{r=0}^d \mrk{K}_r,\quad  \mrk{E}^{(a)}=\oplus_{r=0}^{d} \mrk{E}_{r,r+a},\quad \mrk{F}^{(a)}=\oplus_{r=0}^d \mrk{F}_{r,r-a}.$$
These are endofunctors on $\oplus_{r=0}^d \D(\F_r)$.

\subsection{} 

We fix a sequence ${\bf d}=(d_1,d_2, \cdots, d_m)$ of integers such that $\sum_{l=1}^m d_l=d$. To such a sequence ${\bf d}$,
we associate a fixed  partial flag in $k^d$ of the form
\begin{equation}\label{eq38}
0=V_0 \subset V_1 \subset \cdots \subset V_m=k^d , \quad |V_l/V_{l-1}|=d_l,\ \forall\ l.
\end{equation}
Denote by $P_{\bf d}$ the parabolic subgroup of $G=GL(k^d)$ which fixes all subspace $V_l, \forall\  l=1,\cdots, m$ in the fixed flag (\ref{eq38}).
Let $\mathcal{Q}_{\bf d}^r$ be the full subcategory of $\D(\F_r)$ consisting of $P_{\bf d}$-equivariant semisimple complexes and $\mathcal{Q}_{\bf d}=\oplus_r \mathcal{Q}_{\bf d}^r$.
It is clear that  the functors $\mrk K^{\pm 1}$, $\mrk E^{(a)} $ and $\mrk F^{(a)}$ induce functors
\[
\mrk K^{\pm 1}, \mrk E^{(a)}, \mrk F^{(a)}: \mathcal Q_{\bf d}\to \mathcal Q_{\bf d}.
\]
Let $\mathbf Q_{\bf d}$ be the split Grothendieck group of  $\mathcal{Q}_{\bf d}$.  It admits a left $\tilde \A$-module structure similar to (\ref{eq7}).
We set
\[
\mbf V_{\bf d} =\A\otimes_{\tilde \A} {\bf Q}_{\bf d} \quad \mbox{and}\quad
\mathsf{V}_{\bf d}=\Q[t^{\pm 1}](v)\otimes_{\A} \mathbf{V}_{\bf d}.
\]
The functors $\mrk  K^{\pm 1}$, $\mrk E^{(a)}$, and $\mrk F^{(a)}$ descend to linear maps
\[
\mbf K^{\pm 1}, \mbf E^{(a)}, \mbf F^{(a)}:  \mathsf{V}_{\bf d} \to \mathsf{V}_{\bf d}.
\]
We also use the same for the respective linear maps on $\mbf V_{\bf d}$.
By Lemma ~\ref{lem3}, we have

\begin{lem}
 The quadruple $( \mathsf{V}_{\bf d}, \mbf K^{\pm 1}, \mbf E,\mbf F)$  defines a $\U$-module.
 It is a $_{\A}\!\U$-module if $\mathsf{V}_{\bf d}$ is replaced by $\mbf V_{\bf d}$.
\end{lem}

Moreover, the quadruple ($\mbf V_{\bf d}, \mbf K^{\pm 1}, \mbf E^{(a)}, \mbf F^{(a)}$) defines a module of the integral form of $\U$.
We shall show that $ \mathsf{V}_{\bf d} \equiv (\mathsf{V}_{\bf d}, \mbf K^{\pm 1}, \mbf E,\mbf F) $ is isomorphic to the tensor product module $\Lambda_{\bf d}$ in Section \ref{super} (\ref{Ld}).

\subsection{}
In this subsection, we treat the special case ${\bf d}=(d)$.
In particular, the category  $\mathcal{Q}_{\bf d}$  is nothing but the category $\mathcal Q_d^{r, 0}$ in Section \ref{sec3.5}.
Again, by Proposition  \ref{prop4} and Lemma \ref{lem4}, we have

\begin{prop}\label{prop5}
 $\mathsf{V}_d \simeq \Lambda_d$ and $\mbf V_d\simeq {}_{\A}\!\Lambda_d$, as modules of $\U$ and ${}_{\A}\!\U$, respectively. In particular,
 \begin{equation}\label{eq15}
   \begin{split}
 &\mrk{E}\xi_r=t^{r-1}[d+1-r]_{v}\xi_{r-1}, \quad  \mrk{F}\xi_r=t^r[r+1]_{v}\xi_{r+1}, \quad
  \mrk{K}\xi_r=v^{d-2r}t^{2r-1}\xi_{r},
   \end{split}
 \end{equation}
 where
 $\xi_r=(\bar{\mbb Q}_l)_{\F_r}$.
\end{prop}

\begin{proof}
By Proposition  \ref{prop9}, it is enough to show (\ref{eq15}).
Since $\xi_r=\Psi_{r,0}(F_{r,0})(\bar{\mathbb{Q}}_l)_{\F_0}$, by Proposition  \ref{prop4}, it remains to show that
\begin{equation} \label{eq100}
    \begin{split}
      K  F_{r,0}=v^{d-2r}t^{2r}F_{r,0},\;
   F  F_{r,0}=t^r[r+1]_{v}F_{r+1,0},
   \ {\rm and}\  E  F_{r,0}=t^{r-1}[d+1-r]_{v}F_{r-1,0}.
    \end{split}
  \end{equation}
  The first two identities in (\ref{eq100}) follows from Lemma \ref{sl-relation} and Lemma \ref{lem1}.
We now show the last one in (\ref{eq100}). By Lemma \ref{sl-relation}, we have
\begin{eqnarray*}
E  F_{r,0}&= \sum_{r'} E_{r',r'+1} F_{r,0} =E_{r-1,r} F_{r,0}.
\end{eqnarray*}
Let us compute the complex $E_{r-1,r} F_{r,0}$.
The support of $p_{12}^* E_{r-1,r} \otimes p_{23}^* F_{r,0}$ is
\[
S=\{ (V, V', V'')\in \F_{r-1}\times \F_{r} \times \F_{0}| V'' \subset V' \supset V\}.
\]
Observe that $p_{13}$ is a fiber bundle of fiber isomorphic to  the projective space $\mbb P^{d-r}$.
 Further, the image of $p_{13}$ is $\F_{r-1, 0}$.
The restriction of $p_{12}^* E_{r-1,r} \otimes p_{23}^* F_{r,0}$ to $S$ is
\[
(\bar{\mbb Q}_l)_S [ d-r ](\frac{d-1}{2})\otimes (\bar{\mbb Q}_l)_S =(\bar{\mbb Q}_l)_S  [d-r]\left(\frac{d-1}{2}\right).
\]
So
\begin{equation*}
\begin{split}
E_{r-1,r} F_{r,0}
&=(p_{13})_! (p_{12}^* E_{r-1,r} \otimes p_{23}^* F_{r,0})
=(p_{13})_! ( \bar{\mbb Q}_l)_S  [d-r]\left(\frac{d-1}{2}\right)\\
&= \oplus_{k=0}^{d-r} (\bar{\mbb Q}_l)_{\F_{r-1,0}} [d-r-2k]\left(\frac{d-1}{2}-k\right) =t^{r-1} [d-r+1]_{v}F_{r-1,0}.
\end{split}
\end{equation*}
The proposition follows.
\end{proof}

\subsection{}
In this subsection, we treat the general case.
 Let
 \[
 \Xi^{\bf d}_r=\{ {\bf r}=(r_1,r_2, \cdots, r_m)\in \mbb{Z}_{\geq 0}^m | r=\sum_{1\leq l\leq 1} r_l, r_l\leq d_l, \forall l\}
 \quad\mbox{and}\quad
 \Xi^{\bf d}=\cup_{0\leq r\leq d} \Xi^{\bf d}_r.
 \]
To each ${\bf r}\in \Xi^{\bf d}_r$, we associate a $P_{\bf d}$-orbit  in  $\F_r$ as follows.
 $$O_{\bf r}=\{W \in \F_r\ |\ |W \cap V_l/W \cap V_{l-1}|=r_l, \forall\ l\}.$$
The intersection complexes $IC(O_{\bf r})$ are all possible $P_{\bf d}$-equivariant simple perverse sheaves on $\F_r$.
  So we have

\begin{lem}
\label{V-basis}
The intersection complexes $IC(O_{\bf r}),\ \forall\ {\bf r}\ \in \Xi^{\bf d}$, form an $\A$-basis of $\mathbf{V}_{\bf d}$ and a $\Q(v)[t]$-basis of $\mathsf{V}_{\bf d}$.
\end{lem}

Next, we want to define the restriction functor ``${\rm Res}$''.
In the following, we use notation $\F^d_r$ instead of $\F_r$ to avoid ambiguities.
We fix pairs $(r', r'')$ and $(d', d'')$ of nonnegative integers such that $r'+r''=r$ and $d'+d''=d$.
We  fix a vector subspace $W$ in $k^d$ such that $|W|=d'$. Consider the following diagram
 \begin{equation}\label{eq37}
\xymatrix{\F_{r'}^{d'}\times \F_{r''}^{d''} & Y^{r',r''} \ar[l]_-{\kappa} \ar[r]^{\iota} &\F_r^{d},
 }
 \end{equation}
 where $Y^{r',r''}=\{W' \in \F_r^d\ |\ |W\cap W'|=r'\}$, $\kappa(W')=(W \cap W', W'/(W \cap W'))$ and $\iota$ is the closed embedding. We define
 \begin{align*}\label{eq8}
   {\rm Res}_{ d',d''}^{r',r''}&=\kappa_!\iota^*[(d''-r'')r']\left ( \frac{d''r'}{2} \right ): \D(\F_r^{d})\rightarrow \D( \F_{r'}^{d'}\times \F_{r''}^{d''}),\quad\mbox{and}\\
   {\rm Res}_{ d',d''} & =\oplus_{r',r''}{\rm Res}_{ d',d''}^{r',r''}: \bigoplus_{0\leq r\leq d} \D(\F_r^{d})\rightarrow \bigoplus_{0\leq r'\leq d' ,0\leq r''\leq d''}\D( \F_{r'}^{d'}\times \F_{r''}^{d''}).
 \end{align*}
 We define
 \begin{align*}
& \mathfrak{E}'_{r,r+1}=\mathfrak{E}_{r',r'+1}\times Id: \quad  \D(\F_{r+1}) \otimes \oplus_{r'}\D(\F_{r'})\rightarrow \D(\F_{r}) \otimes \oplus_{r'}\D(\F_{r'}),\\
&  \mathfrak{E}''_{r,r+1}=Id\times \mathfrak{E}_{r,r+1}:  \quad \oplus_{r'}\D(\F_{r'})\otimes \D(\F_{r+1})\rightarrow \oplus_{r'}\D(\F_{r'}) \otimes \D(\F_{r}).
\end{align*}
Similarly, we define the  notations $\mathfrak{K}'_r, \mathfrak{F}'_{r,r-1}$, $\mathfrak{K}''_r$ and $ \mathfrak{F}''_{r,r-1}$.
The following  proposition is a mixed version of   Proposition 3.8.3 in \cite{Zheng07}.

 \begin{prop}\label{prop7}
 For any $ C_r\in \mathcal{Q}^r_{\bf d}$, we have
 \begin{equation}\label{eq10}
 \begin{split}
  &  {\rm Res}_{ d',d''}^{r_1,r_2}\mathfrak{K}_rC_r = \mathfrak{K}'_{r_1}\mathfrak{K}''_{r_2}{\rm Res}_{ d',d''}^{r_1,r_2}C_r,\quad \\
   &   {\rm Res}_{ d',d''}^{r_1,r_2}\mathfrak{E}_{r-1,r}C_r = \mathfrak{E}'_{r_1,r_1+1}{\rm Res}_{ d',d''}^{r_1+1,r_2}C_r\oplus \mathfrak{K}'_{r_1}\mathfrak{E}''_{r_2,r_2+1}(r_1){\rm Res}_{ d',d''}^{r_1,r_2+1}C_r, \\
 & {\rm Res}_{ d',d''}^{r_1,r_2}\mathfrak{F}_{r+1,r}C_r = \mathfrak{F}'_{r_1,r_1-1}\mathfrak{K}''^{-1}_{r_2}{\rm Res}_{ d',d''}^{r_1-1,r_2}C_r\oplus\mathfrak{F}''_{r_2,r_2-1}(r_1){\rm Res}_{ d',d''}^{r_1,r_2-1}C_r.
 \end{split}
 \end{equation}
 \end{prop}

  \begin{proof}
 The first identity is obvious.
We now prove the second one.  We only need to prove  it for $C_r$ a simple perverse sheaf.

We claim that for any simple perverse sheaf  $ C_r\in \mathcal{Q}^r_{\bf d}$, there exists a proper map $\pi:\widetilde{\F}\rightarrow \F_r$, where $\tF$ is
 a smooth irreducible variety, such that $C_r$ is a direct summand of $\pi_!(\bar{\mathbb{Q}}_l)_{\widetilde{\F}}$ up to a shift.
 Let $\widetilde{\iota}: \F_r \rightarrow \F_r\times \F_{d_1}$ be the embedding map sending $V'\mapsto (V', V_1)$, where $V_1$ is the fixed vector space in (\ref{eq38}).
 By the argument in ~\cite[Section 2.6.2]{BL94}, the functor $\widetilde{\iota}^*: \D_{P_{\widetilde{\bf d}}}( \F_r\times \F_{d_1})\rightarrow \D_{P_{\bf d}}(\F_r)$ is an equivalence, where $\widetilde{\bf d}=(d_2, \cdots,d_m)$.
 By the argument in Section \ref{sec3.4}, for any object $C'$ in  $\D_{P_{\widetilde{\bf d}}}( \F_r\times \F_{d_1})$, there exist a smooth irreducible variety $\widetilde{\F}'$ and a proper map $\pi':\widetilde{\F}'\rightarrow \F_r\times \F_{d_1}$  such that $C'$ is a direct summand of $\pi'_!(\bar{\mathbb{Q}}_l)_{\widetilde{\F}'}$.
 Let $\widetilde{\F}=\F_r\times_{(\F_r\times \F_{d_1})}\widetilde{\F}'$ and $C_r=\widetilde{\iota}^*C'$. By base change formula \ref{perversesheaf} (8),  $C_r$ is a direct summand of $\pi_!(\bar{\mathbb{Q}}_l)_{\widetilde{\F}}$, where $\pi: \widetilde{\F}\rightarrow \F_r$ is the pull back map of $\pi'$. This proves the claim.

 By the above claim, we may and will assume that $C_r=\pi_!(\bar{\mathbb{Q}}_l)_{\widetilde{\F}}$ for some smooth irreducible variety $\widetilde{\F}$ and a proper map $\pi:\widetilde{\F}\rightarrow \F_r$.
Now consider the following diagram
$$\xymatrix{
&&&&X' \ar[dll]_{\iota'} \ar[drr]^s  &&&&\\
\F_r && \F_{r-1,r} \ar[ll]_-{p'} \ar[rr]^-{p} && \F_{r-1} && X'' \ar[ll]_-{\iota''} \ar[rr]^-{\kappa''} && \F^{d'}_{r_1}\times \F^{d''}_{r_2}}$$
where $X''=\{V\in \F_{r-1}\ |\ |V\cap W|=r_1\}$ and $X'=\F_{r-1,r}\times_{\F_{r-1}}X''$. By base change formula \ref{perversesheaf} (8), we have
$\kappa''_!\iota''^*p_!p'^*C_r= \kappa''_!s_!\iota'^*p'^*C_r.$

Let $X=\{V'\in \F_r\ |\ |V'\cap W|=r_1\ {\rm or}\ r_1+1\}$. Then $X$ has a partition $X=X_1\sqcup X_2$ with $X_j=\{V'\in \F_r\ |\ |V'\cap W|=r_1+j-1\}$ for $j=1,2$. Let $Y=Y_1\sqcup Y_2$ with $Y_1=\F_{r_1}\times \F_{r_2+1}$ and
$Y_2=\F_{r_1+1}\times \F_{r_2}$. Let $Y'=Y'_1\sqcup Y'_2$ with $Y'_1=\F_{r_1}\times \F_{r_2, r_2+1}$ and $Y'_2=\F_{r_1, r_1+1}\times \F_{r_2}$.
Let $Z_j=X_j\times_{Y_j}Y'_j$ and $Z=Z_1 \sqcup Z_2$. Consider the following diagram
$$\xymatrix{\widetilde{\F}\ar[d]_{\pi} &&\widetilde{Z} \ar[ll] \ar[d]^{\widetilde{\pi}'} && \widetilde{X'} \ar[ll]_{\widetilde{b}} \ar[d]^{\widetilde{\pi}}\\
 \F_r & &Z \ar[d]_-{\kappa'} \ar[ll]_-{\iota s'}\ar[dll]_-{ s'} && X' \ar[ll]_b \ar[d]^{\kappa'' s}\\
 X\ar[u]_{\iota}\ar[r]_{\kappa}& Y & Y' \ar[l]_{q'} \ar[rr]^{q} && \F^{d'}_{r_1}\times \F^{d''}_{r_2}
}$$
where $b$ is the map such that $p'\iota'=\iota s' b$, $\widetilde{Z}=Z\times_{\F_r}\widetilde{\F}$ and
$\widetilde{X'}=X'\times_{\F_r}\widetilde{\F}$.  By the base change formula, the commutativity of the diagram and the assumption $C_r=\pi_!(\bar{\mathbb{Q}}_l)_{\widetilde{\F}}$, we have
$$ \kappa''_!s_!\iota'^*p'^*C_r=q_! \kappa'_!b_!b^*s'^*\iota^*C_r=q_! \kappa'_!b_!\widetilde{\pi}_!(\bar{\mathbb{Q}}_l)_{\widetilde{X'}}.$$

We further have a partition of $\widetilde{X'}=\widetilde{X'_1}\sqcup \widetilde{X'_2}$ and $\widetilde{Z}=\widetilde{Z_1}\sqcup \widetilde{Z_2}$.
For a map in the above diagram, if its domain has a partition, we shall use a subscript $j$ ($j=1,2$) to indicate the restriction map to the corresponding part.
We can check that $b_1$ is a vector bundle of rank $r_1$ (resp. $b_2$ is an identity map), so is $\widetilde{b}_1$ (resp. $\widetilde{b}_2$).
By ~\cite[8.1.6]{Lusztig93}, we have
\begin{eqnarray*}
&&b_!b^*s'^*\iota^*C_r=b_!\widetilde{\pi}_!(\bar{\mathbb{Q}}_l)_{\widetilde{X'}}
=b_{1!}\widetilde{\pi}_{1!}(\bar{\mathbb{Q}}_l)_{\widetilde{X'_1}}\oplus b_{2!}\widetilde{\pi}_{2!}(\bar{\mathbb{Q}}_l)_{\widetilde{X'_2}}\\
&=&\widetilde{\pi}'_1(\bar{\mathbb{Q}}_l)_{\widetilde{Z}_1}[-2r_1]
\oplus \widetilde{\pi}'_2(\bar{\mathbb{Q}}_l)_{\widetilde{Z}_2}
=s_1'^*\iota_1^*C_r[-2r_1]\oplus s_2'^*\iota_2^*C_r,
\end{eqnarray*}
If we ignore  Tate twists, then we have
\begin{eqnarray*}
  &&{\rm Res}_{ d',d''}^{r_1,r_2}\mathfrak{E}_{r-1,r}C_r  =
  q_{1!}\kappa_{1!}'s'^*_1\iota^*_1C_r[N-2r_1]\oplus q_{2!}\kappa_{2!}'s'^*_2\iota^*_2C_r[N]\\
  &= &\mathfrak{K}'_{r_1}\mathfrak{E}''_{r_2,r_2+1}{\rm Res}_{ d',d''}^{r_1,r_2+1}C_r[N-2r_1-N_1]\oplus \mathfrak{E}'_{r_1,r_1+1}{\rm Res}_{ d',d''}^{r_1+1,r_2}C_r[N-N_2],
\end{eqnarray*}
where $N=d-r+(d''-r_2)r_1$,
$N_1=d'-2r_1+d''-r_2-1+(d''-r_2-1)r_1$ and $N_2=d'-r_1-1+(d''-r_2)(r_1+1)$. By using $d=d'+d''$ and $r_1+r_2=r-1$, we have $N-2r_1-N_1=0=N-N_2$. This shows that the complexes on both sides in the second identity are isomorphic to each other if the Tate twists are ignored.

Next, we check that the weights of complexes on both sides in the second identity are the same.
 By \cite{B03}, we see that  $\kappa_!\iota^*$ in Diagram (\ref{eq37}) is equivalent to a hyperbolic localization functor.
 By ~\cite[Theorem 8]{B03}, the functor $\kappa_!\iota^*$ preserves purities and weights of equivariant complexes. Thus
  \begin{equation}\label{eq27}
{\rm wt}({\rm Res}_{ d',d''}^{r_1,r_2}\mathfrak{E}_{r-1,r}(\bar{\mathbb{Q}}_l)_{\F_{r}})=-r+1-r_1r_2.
  \end{equation}
    On the other hand, we have
  \begin{equation}\label{eq28}
  \begin{split}
& {\rm wt}(\mathfrak{E}'_{r_1,r_1+1}{\rm Res}_{ d',d''}^{r_1+1,r_2}(\bar{\mathbb{Q}}_l)_{\F_{r}})=-(r_1+1)r_2-r_1,\\
& {\rm wt}(\mathfrak{K}'_{r_1}\mathfrak{E}''_{r_2,r_2+1}(r_1){\rm Res}_{d',d''}^{r_1,r_2+1}(\bar{\mathbb{Q}}_l)_{\F_{r}})=-r_1(r_2+1)-r_2-4r_1.
  \end{split}
  \end{equation}
  (\ref{eq27}) and (\ref{eq28}) are the same by using $r_1+r_2=r-1$  and $4r_1 \equiv 0\ {\rm mod}\ 4$. This shows that the second equality holds.
  The third one can be proved similarly.
 \end{proof}

Let ${\bf d}'=(d_1,d_2, \cdots, d_{m'})$ and ${\bf d}''=(d_{m'+1},d_{m'+2}, \cdots, d_m)$. Let $d'=\sum_{l=1}^{m'}d_l$ and  $d''=\sum_{l=m'+1}^md_l$ so that  $d'+d''=d$.
Let $\mathcal{Q}_{\bf d',d''}^{r',r''}$ be the full subcategory of $\D(\F_{r'}^{d'}\times \F_{r''}^{d''})$ consisting of all $P_{\bf d'}\times P_{\bf d''}$-equivariant semisimple complexes and $\mathcal{Q}_{\bf d',d''}=\oplus \mathcal{Q}_{\bf d',d''}^{r',r''}$.  It is clear from Proposition ~\ref{prop7} that ${\rm Res}_{ d',d''}$ restricts to a functor
\[
{\rm Res}_{ d',d''}: \mathcal Q_{\mbf d} \to \mathcal{Q}_{\bf d',d''}.
\]
Let $\mathbf{Q}_{\bf d',d''}$ be the split Grothendieck group of $\mathcal{Q}_{\bf d',d''}$ and $\mbf V_{\bf d', d''}=\A\otimes_{\tilde\A} \mbf Q_{\bf d',d''}$.
From the definitions, we have  that $\mbf V_{\bf d', d''} \simeq \mbf V_{\bf d'}\otimes \mbf V_{\bf d''}$.
The functor ${\rm Res}_{ d',d''}$ induces an $\A$-linear map
 $$r_{ d',d''}: \mathbf{V}_{\bf d} \rightarrow \mathbf{V}_{\bf d'}\otimes \mathbf{V}_{\bf d''}.$$
An argument similar to the proof of Proposition 3.8.1 in \cite{Zheng07}\label{lem7} shows that
 $r_{ d',d''}$ is an $\A$-linear isomorphism. Moreover, we have

 \begin{thm}
 \label{thm1}
(a) The maps $r_{d', d''}$ induce  isomorphisms $\mathsf{V}_{\bf d}\simeq \Lambda_{\bf d}$ and $\mbf V_{\bf d} \simeq  {}_{\A}\!\Lambda_{\bf d}$  of  modules of $\U$ and $_{\A}\!\U$, respectively.

 (b)
 The image of $\{IC(O_{\bf r})\ |\ {\bf r}\in \Xi^{\bf d}\}$ under the above isomorphism  form a $\Q[t](v)$-basis of $\Lambda_{\bf d}$.
 Moreover, the structure constants of the actions of $E$, $F$ and $K$ on $ IC(O_{\bf r})$ are in $t^a \mathbb N[v^{\pm 1}]$ for various $a\in\mathbb Z$ with respect to this  basis.

 \end{thm}

 \begin{proof} (a) follows from Propositions \ref{prop5}, \ref{prop7} and  (\ref{eq44}).
 The first statement of part (b) follows from the fact that ${\bf r}\in \Xi^{\bf d}_r$ parameterizes the  $P_{\bf d}$-orbits of $\F_{r}$. The second statement of (b) follows from the first one and Lemma ~\ref{V-basis}.
 \end{proof}

 \subsection{} \label{sec4.10}
 Recall that the ${\rm Ext}$ groups of any two objects $L, M$ in $\D(X)$ are defined by
 $${\rm Ext}^n(L,M)={\rm Hom}_{\D(X)}(L, M[n]).$$
 We will use the following properties of ${\rm Ext}$ groups.
 \begin{itemize}
   \item [(a)] ${\rm Ext}^j(L[n], M[m])={\rm Ext}^{j-n+m} (L, M)$;

   \item [(b)] If both $L$ and $M$ are perverse sheaves, then ${\rm Ext}^j_{\D(X)}(L, M)=0$ for any $j <0$;

   \item [(c)]  Suppose that $L$ and $ M$ are both simple perverse sheaves, then $\dim {\rm Ext}^0(L, M)=1$ if $L= M$ and 0 otherwise.
 \end{itemize}

 Given any two pure complexes $L, M$ in $\mathcal{Q}_{\bf d}$, we define
 \begin{equation}\label{eq30}
   (L, M)=\sum_{j\in \mbb{Z}} \dim {\rm Ext}^j(L, \mbb{D}M)v^{-j}t^{-{\rm wt}(L)-{\rm wt}(M)}.
 \end{equation}
 Since any complex in $\mathcal{Q}_{\bf d}$ is semisimple, the above definition can be extended to any two complexes in $\mathcal{Q}_{\bf d}$. This defines a bilinear form on $\mathbf{V}_{\bf d}$.
 \begin{prop}\label{prop12}
   For any two complexes $L,M$ in $\mathcal{Q}_{\bf d}$, we have
   \begin{equation*}
     \begin{split}
       (\mathfrak{K}_rL, M)&=(L, \mathfrak{K}_rM),\;\\
       (\mathfrak{E}_{r,r+1}L, M) & =(L, \mathfrak{K}_{r+1}\mathfrak{F}_{r+1,r}M[1](-\frac{2r+1}{2})),\\
      (\mathfrak{F}_{r+1,r}L, M)&=(L, \mathfrak{K}_{r}^{-1}\mathfrak{E}_{r,r+1}M[1](\frac{2r+1}{2})).
     \end{split}
   \end{equation*}
 \end{prop}
 \begin{proof}The first equality is obvious. We now show the second equality and  the third one can be proved similarly. By the definition of $\mathfrak{E}_{r,r+1}$ in (\ref{eq31}), we have
 \begin{eqnarray*}
   &{\rm Ext}^j(p_!p'^*L[d-1-r](\frac{d-1}{2}), \mbb{D}M)={\rm Ext}^j(L, p'_*p^![-d+1+r](-\frac{d-1}{2})\mbb{D}M)\\
   &={\rm Ext}^j(L, \mbb{D}(p'_!p^*[d-1-r](\frac{d-1}{2})M))={\rm Ext}^j(L, \mbb{D}(\mathfrak{K}_{r+1}\mathfrak{F}_{r+1,r}[1]M)).
 \end{eqnarray*}
Without loss of generality, we assume that both $L$ and $M$ are pure  complexes. By (\ref{eq30}), we have
\begin{equation*}
\begin{split}
   (\mathfrak{E}_{r,r+1}L, M)&=\sum_{j\in \mbb{Z}} \dim {\rm Ext}^j(\mathfrak{E}_{r,r+1}L, \mbb{D}M)v^{-j}t^{-{\rm wt}(L)-{\rm wt}(M)+r}\\
   &=\sum_{j\in \mbb{Z}} \dim {\rm Ext}^j(L, \mbb{D}(\mathfrak{K}_{r+1}\mathfrak{F}_{r+1,r}[1]M))v^{-j}t^{-{\rm wt}(L)-{\rm wt}(M)+r}\\
&=(L, \mathfrak{K}_{r+1}\mathfrak{F}_{r+1,r}M[1](-\frac{2r+1}{2})).
\end{split}
\end{equation*}
The proposition follows.
 \end{proof}

  We define an algebra isomorphism $\rho: \mbf S_{v,t}(2,d)\rightarrow (\mbf S_{v,t}(2,d))^{op}$ by
 $$\rho(1_r)=1_r,\quad \rho(E_{r,r+1})=vt^{-2r-2}K_rF_{r+1,r},\quad \rho(F_{r+1,r})=vt^{2r}K_rE_{r,r+1}.$$
The following   corollary follows directly from Proposition \ref{prop12}.

\begin{cor}
For any two isomorphism classes $L$ and $M$ in $\mathbf{V}_{\bf d}$ and any $x \in \mbf  S_{v,t}(2,d)$, we have $(x L, M)=(L, \rho(x)M)$.
\end{cor}

 Given any pure complex $L \in \mathcal{Q}_{\bf d}$, let
$$\mathfrak{D}(L)=(\mathbb{D}L)(-{\rm wt}(L)).$$
 Since objects in $\mathcal{Q}_{\bf d}$ are semisimple, this defines a functor $\mathfrak{D}: \mathcal{Q}_{\bf d} \rightarrow \mathcal{Q}_{\bf d}$.
We notice that $\mathfrak{D}^2$ is the identity functor.
Let $^-: \mathbf{V}_{\bf d}\rightarrow \mathbf{V}_{\bf d}$ be the $\mbb{Z}[t]$-linear map defined by  $L\mapsto \mathfrak{D}L$ and $v\mapsto v^{-1}$. Let $\mathbf{B}_{\bf d}$ be the subset of $\mathbf{V}_{\bf d}$ consisting of all $x$ satisfying
$$\overline{x}=x, \quad (x,x)\in 1+v^{-1}\mbb{Z}[v^{-1}].$$
Recall that a signed basis of a module $M$ is a subset, say $B$, of $M$ such that $B=B'\cup (-B')$ for some basis $B'$ of $M$.
\begin{prop}
  $\mathbf{B}_{\bf d}$ is the canonical signed basis of $\mathbf{V}_{\bf d}$.
\end{prop}
\begin{proof} For any $x\in \mathbf{B}_{\bf d}$, let $L$ be a representative complex of $x$. Then we have $\mathfrak{D}(L)\simeq L$. Let $L\simeq \oplus_{i=1}^m L'_i$, where $L'_i$ are all simple complexes. For each $L'_i$, there exists $a_i\in \mathbb{Z}$ such that $L_i:=L'_i[-a_i](-\frac{1}{2}a_i)$ is simple perverse sheaf. Let $a'=\max_i{a_i}$. Denote by $x_i$ the isomorphic classes of $L_i$ for each $i$. We have $(x,x)=\sum_{i, j}v^{a_i+a_j}(x_i,x_j)$. Since $(x,x)\in 1+v^{-1}\mbb{Z}[v^{-1}]$, by Section \ref{sec4.10} (b), we have $a' \leq 0$. Hence $a_i\leq 0$ for all $i$.

On the other hand, $\mathfrak{D}(L)=\oplus_i \mathfrak{D}(L_i)[-a_i](-\frac{1}{2}a_i)$ and $\mathfrak{D}(L_i)$ is still a simple perverse sheaf. Let $a''=\min_i{a_i}$. By a similar argument, we have $-a''\leq 0$. Hence $a_i \geq 0$ for all $i$. So $a_i=0$ for all $i$ and, therefore,
 $L=\oplus_{i=1}^mL_i$ is a perverse sheaf. By Section \ref{sec4.10} (c), we have $(L, L)\in m+v^{-1}\mbb{Z}[v^{-1}]$. Hence $L$ is a simple perverse sheaf.

 If the weight of $L$ is odd, then $(L,L) \in -1+v^{-1}\mbb{Z}[v^{-1}]$. It is a contradiction.
 Therefore, $L$ is a simple perverse sheaf of weight $0$ or 2.
\end{proof}

\section{Modified forms of $\U$  and weight modules}

\subsection{}

A $\U$-module $M$ is called a {\it weight module} if there is a decomposition of vector spaces $M=\oplus_{\lambda \in \mbb{Z}}M_{\lambda}^{\pm}$ such that
$$M^{+}_{\lambda}=\{m\in M\ |\ K\cdot m=
v^{\lambda}t^{-\lambda-p(\lambda)}m\},\ {\rm and}\ M^{-}_{\lambda}=\{m\in M\ |\ K\cdot m=-
v^{\lambda}t^{-\lambda-p(\lambda)}m\},$$
where $p(\lambda)=0$ if $\lambda$ is even and 1 otherwise.
The subspaces $M_{\lambda}^{\pm}$ are called the weight spaces of $M$.
Let $\mathcal{C}^+$ (resp. $\mathcal{C}^-$) be the category whose objects are weight modules  of the form $M=\oplus_{\lambda} M_{\lambda}^+$ (resp. $M=\oplus_{\lambda} M_{\lambda}^-$) of $\U$ and morphisms are $\U$-linear maps. 

 The {\it modified quantum superalgebra} $\dot{\U}$ associated to $\mathfrak{osp}(1|2)$ is defined to be  the associative  $\mbb{Q}[t](v)$-algebra  without unit,
 generated by $1_{\lambda}, E_{\lambda,\lambda-2}$ and $F_{\lambda,\lambda+2}$, $\forall \lambda\in \mbb{Z}$, and subject to the following defining relations.
 \begin{equation}
 \label{eq33.5}
\begin{split}
1_{\lambda} 1_{\lambda'} =\delta_{\lambda, \lambda'} 1_{\lambda}, &\\
  \quad E_{\lambda, \lambda-2 }  1_{\lambda'} =  \delta_{\lambda-2, \lambda'} E_{\lambda, \lambda-2}, \quad \ &\quad
1_{\lambda'}   E_{\lambda, \lambda -2} = \delta_{\lambda', \lambda} E_{\lambda, \lambda-2},\\
\quad F_{\lambda, \lambda+2}  1_{\lambda'} =  \delta_{\lambda+2, \lambda'} F_{\lambda, \lambda+2}, \quad \ &\quad
1_{\lambda'}  F_{\lambda, \lambda+2} =  \delta_{\lambda', \lambda} F_{\lambda, \lambda+2},\\
\end{split}
\end{equation}
\begin{equation}\label{eq34}
  E_{\lambda,\lambda-2} F_{\lambda-2,\lambda}-t^2F_{\lambda,\lambda+2}  E_{\lambda+2,\lambda}=t^{-\lambda-p(\lambda)}[\lambda]_v1_\lambda.
\end{equation}
Let $\dot{\mathcal{C}}$ be the category of unital $\dot{\U}$-modules in the sense of Lusztig ~\cite[23.1.4]{Lusztig93}.
Given a weight $\U$-module $M=\oplus_{\lambda}M^{+}_{\lambda}$, we define a $\dot{\U}$-module structure on $M$ as follows.
 $$
 E_{\lambda'+2, \lambda'}\cdot m= \delta_{\lambda, \lambda'}E\cdot m, \;
F_{\lambda'-2, \lambda'}\cdot m= \delta_{\lambda, \lambda'}F\cdot m\ {\rm and}\;
1_{\lambda'}\cdot m= \delta_{\lambda, \lambda'}m,\; {\rm for\ any}\ m\in M^{+}_{\lambda}.$$
 This $\dot{\U}$-module structure is well-defined.
To prove this,  we only  need to check the  relation (\ref{eq34}). The rest are obvious.
For any $m\in M^{+}_{\lambda}$, we have
\begin{eqnarray*}
&&(E_{\lambda,\lambda-2} F_{\lambda-2,\lambda}-t^2F_{\lambda,\lambda+2} E_{\lambda+2,\lambda})\cdot m=(EF-t^2FE)\cdot m\\
&&=\frac{K-K^{-1}}{v-v^{-1}}m=\frac{v^{\lambda}t^{-\lambda-p(\lambda)}
-v^{-\lambda}t^{\lambda+p(\lambda)}}{v-v^{-1}}m
=t^{-\lambda-p(\lambda)}[\lambda]_v1_{\lambda}m,
\end{eqnarray*}
where the last equality follows from $t^{\lambda+p(\lambda)}=t^{-\lambda-p(\lambda)}$.
It is clear that a homomorphism $f: M\to N$  in $\mathcal C^+$ becomes a homomorphism in $\dot{\mathcal C}$ if $M$ and $N$ are regarded as $\dot \U$-modules.
The above analysis  provides us with  a functor
\begin{equation}\label{eq39}
\eta: \mathcal{C}^+ \rightarrow \dot{\mathcal{C}}.
\end{equation}

Conversely, given a $\dot{\U}$-module $M$, let $M^+_{\lambda}=1_{\lambda}\cdot M$.
By using Lemma  \ref{sl-relation} (a), one can easily show that $M^+_{\lambda}\cap M^+_{\lambda'}=\{0\}$ if $\lambda\neq \lambda'$.
So we have $M=\oplus_{\lambda}M^+_{\lambda}$ as a vector space.  We now define a $\U$-module structure on $M$ by
$E\cdot m=E_{\lambda+2,\lambda}\cdot m,\
K\cdot m= v^{\lambda}t^{-\lambda-p(\lambda)} m$
and $F\cdot m=F_{\lambda-2,\lambda}\cdot m$ for any $m\in M^+_{\lambda}$.
Similarly, we can check that this $\U$-module structure is well-defined.
This defines a functor
\begin{equation}\label{eq40}
  \eta': \dot{\mathcal{C}} \rightarrow \mathcal{C}^+.
\end{equation}
It is clear that $\eta \eta'$ and $\eta' \eta$ are identity functors on $\dot{\mathcal{C}}$ and $\mathcal{C}^+$, respectively.
We have the following proposition.

\begin{prop}
\label{prop14}
The functors $\eta$ and $\eta'$ in (\ref{eq39}) and (\ref{eq40}) establish an isomorphism of categories between $\mathcal{C}^+$ and $\dot{\mathcal{C}}$.
\end{prop}

Note that the notion of an isomorphism of categories is stronger than the notion of an equivalence of categories.
We thank Jon Kujawa  for pointing out this to us.

\subsection{}

Recall that the {\it modified quantum algebra} $\dot{\U}(\mathfrak{sl}(2))$ associated to $\mathfrak{sl}(2)$ is a $\mbb{Q} [t](v)$-algebra without unit,
generated by
$\widetilde{1}_{\lambda}, \widetilde{E}_{\lambda,\lambda-2}$ and $\widetilde{F}_{\lambda,\lambda+2}$, $\forall \lambda\in \mbb{Z}$,
 subjects to the analogous relations of  (\ref{eq33.5}) and the following one.
\begin{equation*}\label{eq32}
 \widetilde{ E}_{\lambda,\lambda-2} \widetilde{F}_{\lambda-2,\lambda}-\widetilde{F}_{\lambda,\lambda+2}  \widetilde{E}_{\lambda+2,\lambda}
 =[\lambda]_v \widetilde 1_\lambda.
\end{equation*}

\begin{thm}
\label{thm3}
(a) The assignments
$\widetilde{E}_{\lambda,\lambda-2}\mapsto t^{\lambda+p(\lambda)}E_{\lambda,\lambda-2}, \ \widetilde{F}_{\lambda,\lambda+2}\mapsto F_{\lambda,\lambda+2}$ and $\widetilde{1}_{\lambda}\mapsto 1_{\lambda}$,  for any $\lambda \in \mbb Z$,
define a unique algebra isomorphism $\varphi: \dot{\U}(\mathfrak{sl}(2)) \rightarrow \dot{\U}$.

(b) The algebra $\dot \U$ is isomorphic to  the algebra in the same notation in ~\cite[Section 6]{CW12} over $\mbb Q[t](v)$.

(c) There is a basis in $\dot \U$ whose structure constants are in $\mbb N[v, v^{-1}]$.
\end{thm}

\begin{proof}   We have
\begin{eqnarray*}
  &&\varphi(\widetilde{ E}_{\lambda,\lambda-2} \widetilde{F}_{\lambda-2,\lambda}-
  \widetilde{F}_{\lambda,\lambda+2}  \widetilde{E}_{\lambda+2,\lambda}-[\lambda]_v\widetilde{1}_\lambda)\\
&=& t^{\lambda+p(\lambda)}(E_{\lambda,\lambda-2} F_{\lambda-2,\lambda}-t^{2}F_{\lambda,\lambda+2} E_{\lambda+2,\lambda})-[\lambda]_v1_\lambda=0.
\end{eqnarray*}
Similarly,  one can show that the other defining relations of $\dot \U(\mathfrak{sl}(2))$ get sent to zero by $\varphi$.
This shows that $\varphi$ is an algebra homomorphism.
Similarly, there is a unique  algebra homomorphism $\varphi':  \dot{\U} \rightarrow \dot{\U}(\mathfrak{sl}(2))$ defined by
$E_{\lambda,\lambda-2}\mapsto t^{-\lambda-p(\lambda)}\widetilde{E}_{\lambda,\lambda-2},$ $ F_{\lambda,\lambda+2}\mapsto \widetilde{F}_{\lambda,\lambda+2}$ and $1_{\lambda}\mapsto \widetilde{1}_{\lambda}$. Clearly, $\varphi \varphi'=Id$ and $\varphi'\varphi=Id$. This finishes the proof of (a).
Statement  ($c$) follows by taking  the basis to be the image of the canonical basis of $\dot{\U}(\mathfrak{sl}(2))$ under the isomorphism in (a).
The commutator relation (\ref{eq34}) can be rewritten as
\[
(t^{\lambda+p(\lambda)} E_{\lambda, \lambda-2} )  (t^{\lambda-1} F_{\lambda-2, \lambda}) -
t^2 (t^{\lambda+1}F_{\lambda, \lambda+2}) (t^{\lambda+2+p(\lambda+2)} E_{\lambda+2, \lambda})
=[\lambda]_{v, t} 1_{\lambda}.
\]
By comparing with the commutator relation for the modified quantum $\mathfrak{osp}(1|2)$ in ~\cite[6.3]{CW12}, we have (b).
\end{proof}

Let $\U(\mathfrak{sl}(2))$ be the quantum algebra associated to $\mathfrak{sl}(2)$ defined over the field $\mathbb Q[t](v)$.
 To avoid any confusion, we shall denote by $\widetilde{E}, \widetilde{F}, \widetilde{K}^{\pm 1}$ the standard generators of $\U(\mathfrak{sl}(2))$.
Recall that a $\U(\mathfrak{sl}(2))$-module $M$ is called a {\it weight module of type 1} if there is a decomposition of vector spaces
$M=\oplus_{\lambda \in \mbb{Z}}M_{\lambda}$ such that $M_{\lambda}=\{m\in M\ |\  \widetilde K\cdot m= v^{\lambda}m\}.$
Let $\mathcal{C}^+(\mathfrak{sl}(2))$ be the category whose objects are weight $\U(\mathfrak{sl}(2))$-modules of type 1
and morphisms are $\U(\mathfrak{sl}(2))$-linear maps.
Similarly, we can define the category $\mathcal{C}^-(\mathfrak{sl}(2))$ of weight modules of type $-1$.
By a similar argument, $\mathcal{C}^+(\mathfrak{sl}(2))$ is equivalent to the category of unital  $\dot{\U}(\mathfrak{sl}(2))$-modules.
By Theorem \ref{thm3}, we have

\begin{prop}
\label{thm4}
The category $\mathcal{C}^+$ is isomorphic  to the category $ \mathcal{C}^+(\mathfrak{sl}(2))$.
\end{prop}

\subsection{}

By a similar argument as the proof of Proposition \ref{prop14}, the category $\mathcal{C}^-$ is equivalent to  $\mathcal{C}^-(\mathfrak{sl}(2))$. Let $\mathcal{C}=\mathcal{C}^+\oplus \mathcal{C}^-$.
Note that the highest weight simple  modules $\Lambda_d^{\pm}$, for all $d\in \mbb{N}$, are objects in $\mathcal{C}$. Let $\mathcal{C}(\mathfrak{sl}(2))=\mathcal{C}^+(\mathfrak{sl}(2))\oplus \mathcal{C}^-(\mathfrak{sl}(2))$. Then we have the following theorem.

 \begin{thm}
   The category $\mathcal{C}$ is isomorphic  to $\mathcal{C}(\mathfrak{sl}(2))$.
 \end{thm}

\subsection{}

By Lemma ~\ref{sl-relation},  there is a unique surjective algebra homomorphism
\[
\phi_d: \overset{.}{\U} \to \mbf S_{v,t}(2,d)
\]
defined by
$E_{\lambda, \lambda-2} \mapsto \left\{\begin{array}{ll}
t^{-\frac{d-p(d)}{2}}E_{r, r+1},& {\rm if}\  \lambda=d-2r  \\
0,& {\rm otherwise},\end{array}\right.
  \ 1_{\lambda} \mapsto  \left\{\begin{array}{ll}1_{r},& {\rm if}\  \lambda =d-2r\\
0,& {\rm otherwise},\end{array}\right.$
and $F_{\lambda, \lambda+2} \mapsto \left\{\begin{array}{ll}t^{-\frac{d-p(d)}{2}}F_{r, r-1},
& {\rm if}\ \lambda=d-2r \\
0,& {\rm otherwise.}\end{array}\right .$
By checking the image of the generators, we have
\[
\psi_{d,d+2}\phi_{d+2}=\phi_{d},
\]
where $\psi_{d, d+2}$ is (\ref{psi}) in Section \ref{S_v}.


\begin{thebibliography}{99999}\frenchspacing

\bibitem[BBD82]{BBD82} A. Beilinson, J. Bernstein, P. Deligne,
         {\em Faisceaux pervers},
         Ast\'{e}risque {\bf 100} (1982).


\bibitem[BLM90]{BLM90} A. Beilinson, G. Lusztig, R. McPherson,
          {\em A geometric setting for the quantum deformation of $GL_n$}, Duke Math. J., {\bf 61} (1990), 655-677.

\bibitem[BL94]{BL94} J. Bernstein, V. Lunts,
         {\em Equivariant sheaves and functors},
         LNM {\bf 1578}, 1994.


\bibitem[B03]{B03}T. Braden,
{\em Hyperbolic localization of intersection cohomology}, Transform. Groups, {\bf 8}
(2003), no. 3, 209-216.




\bibitem[CW12]{CW12} S. Clark, W. Wang, {\em  Canonical Basis for Quantum $\mrk{osp}(1|2)$}, Lett. Math. Phys. {\bf 103} (2013), no. 2, 207-231.


\bibitem[CFLW13]{CFLW} S. Clark, Z. Fan, Y. Li, W. Wang, {\em Quantum supergroups III. Twistors}, arXiv:1307.7056.

\bibitem[D95]{D95} J. Du, {\em A note on quantized Weyl reciprocity at roots of unity}, Algebra Colloq. {\bf 2} (1995), no. 4, 363-372.





\bibitem[EKL11]{EKL11} A. Ellis, M. Khovanov, A. Lauda,
          {\em The odd nilHecke algebra and its diagrammatics}, IMRN, 2013.   arXiv:1111.1320.

\bibitem[EL13]{EL13} A. Ellis, A. Lauda, {\em An odd categorification of $U_q(\mathfrak{sl}_2)$}, arXiv:1307.7816.

\bibitem[FL12]{FL12} Z. Fan, Y. Li, {\em Two-parameter quantum algebras, canonical bases and categorifications}, submitted.  arXiv:1303.2429.


\bibitem[GL92]{GL92}I. Grojnowski, G. Lusztig,
          {\em On bases of irreducible representations of quantum $GL_n$.}
          Kazhdan-Lusztig theory and related topics (Chicago, IL, 1989), 167-174,  Contemp. Math.,  139, Amer. Math. Soc., Providence, RI, 1992.


\bibitem[HW12]{HillWang}
D.~Hill,  W.~Wang, \emph{{Categorification of quantum Kac-Moody
  superalgebras}}, Trans. Amer. Math.
Soc. (to appear), arXiv:1202.2769.

\bibitem[KKO12]{KKO12} S.-J. Kang, M. Kashiwara, S.-J. Oh,
       {\em Supercategorification of quantum Kac-Moody algebras}, arXiv:1206.5933.

\bibitem[KKO13]{KKO13} S.-J. Kang, M. Kashiwara, S.-J. Oh,
              {\em Supercategorification of quantum Kac-Moody algebras II},   arXiv:1303.1916.

\bibitem[KKT11]{KKT} S.-J. Kang, M. Kashiwara,  S. Tsuchioka,
       {\em Quiver Hecke superalgebras}, arXiv:1107.1039.


\bibitem[K91]{Kashiwara91} M. Kashiwara,
        {\em On crystal bases of the q-analogue of universal
          enveloping algebras}, Duke Math J. {\bf 63} (1991), 465-516.

\bibitem[KL09]{Khovanov2009diag}
M.~Khovanov,  A. Lauda, \emph{A diagrammatic approach to categorification
  of quantum groups. {I}}, Represent. Theory \textbf{13} (2009), 309--347.

\bibitem[KL11] {Khovanov2010diagrammatic}
M.~Khovanov,  A. Lauda, \emph{A diagrammatic approach to categorification
  of quantum groups {II}}, Trans. Amer. Math. Soc. \textbf{363} (2011), no.~5,
  2685--2700.

\bibitem[Lan02]{Lan02} E. Lanzman,
          {\em The Zhang Transformation and
              $U_q(osp(1, 2l))$-Verma Modules Annihilators}, Algebras and Representation Theory {\bf 5}: 235-258, 2002.

\bibitem[Lau10]{Lauda10}  A.  Lauda,
          {\em A  categorification of quantum sl(2)},  Adv. Math. {\bf 225} (2010), no. 6, 3327-3424.

\bibitem[Li10]{Li10} Y. Li, {\em A geometric realization of modified quantum algebras},  arXiv:1007.5384.

\bibitem[L90]{Lusztig90} G. Lusztig,
        {\em Canonical bases arising from quantized enveloping
         algebras},
       J. Amer. Math. Soc. {\bf 3} (1990), 447-498.


\bibitem[L91]{Lusztig91} G. Lusztig,
        {\em Quivers, perverse sheaves, and quantized enveloping
          algebras},
        J. Amer. Math. Soc., {\bf 4} (1991), 365-421.

\bibitem[L92]{Lusztig92} G. Lusztig,
        {\em Canonical bases in tensor products},
         Proc. Natl. Acad. Sci. USA {\bf 89} (1992),  8177-8179.

\bibitem[L93]{Lusztig93} G. Lusztig,
        {\em Introduction to Quantum Groups},
        Progress in Math. {\bf 110}, Birkh\"{a}user {1993}.


\bibitem[R06]{R} R. Rouquier,
      {\em Categorification of $\mathfrak{sl}(2)$ and braid groups}.  {\em Trends in representation theory of algebras and related topics}, 137-167, Contemp. Math., 406, Amer. Math. Soc., Providence, RI, 2006.

\bibitem[R08]{R08} R. Rouquier,
     {\em 2-Kac-Moody algebras},  arXiv:0812.5023.

\bibitem[VV11]{VV11} E. Vasserot, M. Varagnolo,
       {\em Canonical bases and KLR algebras}, J. reine angew. Math. {\bf 659} (2011), 67-100.

\bibitem[W09]{W09} W. Wang,
     {\em Double affine Heke algebras for the spin symmetric group},  Math. Res. Lett., {\bf 16} 1071-1085, 2009.
    arXiv:math.RT/0608074.


\bibitem[Z07]{Zheng07} H. Zheng,
       {\em A geometric categorification of tensor products of $U_q(sl_2)$-modules},
  Topology and Physics, 348-356, Nankai Tracts Math.,{\bf 12}, World Sci. Publ., Hackensack, NJ, 2008. arXiv:0705.2630.

\bibitem[Zou98]{Zou98} Y. Zou,
{\em Integrable representations of $U_q(osp(1,2n))$},
J. Pure Appl. Algebra {\bf 130} (1998), no. 1, 99-112.

\end{thebibliography}
\end{document}